\documentclass[a4paper,draft,11pt]{article}

\usepackage{amssymb,amsthm,amsmath}
\usepackage[a4paper,margin=2.5cm]{geometry}
\usepackage{stackrel}

\def\Var{{\rm Var\,}}
\def\E{\mathbb E}
\def\p{\mathbb P}

\def\ii{{\bf i}}
\def\jj{{\bf j}}
\def\kk{{\bf k}}
\def\rr{{\bf r}}

\newcommand*{\ind}[1]{\mathbf{1}_{\{#1\}}}
\newcommand*{\Ind}[1]{\mathbf{1}_{#1}}

\newcommand{\Ent}{\mathrm{Ent}}

\def\bfi{\mathbf{i}}

\newcommand{\ub}[1]{^{(#1)}}
\newcommand{\uu}[1]{^{\underline{#1}}}

\newcommand{\D}{\mathbf{D}}
\newcommand{\R}{\mathbb{R}}
\newcommand{\N}{\mathbb{N}}
\newcommand{\id}{\mathrm{Id}}

\newtheorem{lemma}{Lemma}[section]

\newtheorem{theorem}[lemma]{Theorem}
\newtheorem{prop}[lemma]{Proposition}
\newtheorem{cor}[lemma]{Corollary}

\def\norm#1{\left\| #1 \right\|}
\def\dd{{\bf d}}

\def\ee{{\bf e}}

\title{Concentration inequalities for non-Lipschitz functions with bounded derivatives of higher order\thanks{Research partially supported by the Polish Ministry of Science and Higher Education Iuventus Plus Grant no. IP 2011 000171.}}
\author{Rados{\l}aw Adamczak\thanks{Institute of Mathematics. University of Warsaw. Banacha 2, 02-097 Warszawa, Poland.}, Pawe{\l} Wolff\footnotemark[2] \thanks{Institute of Mathematics. Polish Academy of Sciences. \'Sniadeckich 8, 00-956 Warszawa, Poland.}}

\begin{document}
\maketitle
\begin{abstract}
Building on the inequalities for homogeneous tetrahedral polynomials in independent Gaussian variables due to R. Lata{\l}a we provide a concentration inequality for non-necessarily Lipschitz functions $f\colon \R^n \to \R$ with bounded derivatives of higher orders, which hold when the underlying measure satisfies a family of Sobolev type inequalities
\begin{displaymath}
\|g- \E g\|_p \le C(p)\|\nabla g\|_p.
\end{displaymath}
 Such Sobolev type inequalities hold, e.g., if the underlying measure satisfies the log-Sobolev inequality (in which case $C(p) \le C\sqrt{p}$) or the Poincar\'e inequality (then $C(p) \le Cp$). Our concentration estimates are expressed in terms of tensor-product norms of the derivatives of $f$.

 When the underlying measure is Gaussian and $f$ is a polynomial (non-necessarily tetrahedral or homogeneous), our estimates can be reversed (up to a constant depending only on the degree of the polynomial). We also show that for polynomial functions, analogous estimates hold for
arbitrary random vectors with independent sub-Gaussian coordinates.

 We apply our inequalities to general additive functionals of random vectors (in particular linear eigenvalue statistics of random matrices) and the problem of counting cycles of fixed length in Erd\H{o}s-R{\'e}nyi random graphs, obtaining new estimates, optimal in a certain range of parameters.

\medskip


\end{abstract}

\section{Introduction}
Concentration of measure inequalities are one of the basic tools in modern probability theory (see the monograph \cite{LedouxConcBook}). The prototypic result for all concentration theorems is arguably the Gaussian concentration inequality \cite{BorellGaussianConc,SCGaussianConc}, which asserts that if $G$ is a standard Gaussian vector in $\R^n$ and $f\colon \R^n \to \R$ is a 1-Lipschitz function, then for all $t > 0$,
\begin{displaymath}
\p(|f(G) - \E f(G)|\ge t) \le 2\exp(-t^2/2).
\end{displaymath}

Over the years the above inequality has found numerous applications in the analysis of Gaussian processes, as well as in asymptotic geometric analysis (e.g. in modern proofs of Dvoretzky type theorems). Its applicability in geometric situations comes from the fact that it is dimension free and all norms in $\R^n$ are Lipschitz with respect to one another. However, there are some probabilistic or combinatorial situations, when one is concerned with functions that are not Lipschitz. The most basic case is the probabilistic analysis of polynomials in independent random variables, which arise naturally, e.g., in the study of multiple stochastic integrals, in discrete harmonic analysis as elements of the Fourier expansions on the discrete cube or in numerous problems of random graph theory, to mention just the famous subgraph counting problem \cite{JaRuInf, JanOleRu,ChatterjeeTr,DeMarcoKahnTr,DeMarcoKahnCl}.

The concentration of measure or more generally integrability properties for polynomials have attracted a lot of attention in the last forty years. In particular Bonami \cite{Bon} and Nelson~\cite{N} provided hypercontractive estimates (Khintchine type inequalities) for polynomials on the discrete cube and in the Gauss space, which have been later extended to other random variables  by Kwapie\'{n} and Szulga \cite{KwapienSzulga} (see also \cite{KwapienWoyczynski}). Khintchine type inequalities have been also obtained in the absence of independence for polynomials under log-concave measures by Bourgain \cite{BourgainConvPolynomials}, Bobkov \cite{BobkovPolynomials}, Nazarov-Sodin-Volberg \cite{NazarovSodinVolberg} and Carbery-Wright \cite{CarberyWright}.

Another line of research is to provide two sided estimates of moments of polynomials in terms of deterministic functions of the coefficients.
Borell \cite{Bo} and Arcones-Gin\'{e} \cite{AG} provided such two sided bounds for homogeneous polynomials in Gaussian variables. They were expressed in terms of expectations of suprema of certain empirical processes. Talagrand \cite{TalagrandNewConc} and Bousquet-Boucheron-Lugosi-Massart \cite{BLMEntropy, BBLM} obtained counterparts of these results for homogeneous tetrahedral\footnote{A multivariate polynomial is called tetrahedral if all variables appear in it in power at most one.} polynomials in Rademacher variables and {\L}ochowski \cite{Loch} and Adamczak \cite{AdLogSobConv} for random variables with log-concave tails. Inequalities of this type, while implying (up to constants) hypercontractive bounds, have a serious downside as the analysis of the empirical processes involved is in general difficult. It is therefore important to obtain two-sided bounds in terms of purely deterministic quantities. Such bounds for random quadratic forms in independent symmetric random variables with log-concave tails have been obtained by Lata{\l}a \cite{L1} (the case of linear forms was solved earlier by Gluskin and Kwapie\'n in \cite{GK}, whereas bounds for quadratic forms in Gaussian variables were obtained by Hanson-Wright \cite{HansonWright}, Borell \cite{Bo} and Arcones-Gin\'{e} \cite{AG}). Their counterparts for multilinear forms of arbitrary degree in nonnegative random variables with log-concave tails have been derived by Lata{\l}a and {\L}ochowski \cite{LL}. As for the symmetric case, the general problem is still open. An important breakthrough has been obtained by Lata{\l}a \cite{L2}, who proved two-sided estimates for Gaussian chaoses of arbitrary order, that is for homogeneous tetrahedral polynomials of arbitrary degree in independent Gaussian variables (we recall his bounds below as they are the starting point for our investigations). For general symmetric random variables with log-concave tails similar bounds are known only for chaoses of order at most three \cite{Chaos3d}. 

Polynomials in independent random variables have been also investigated in relation with combinatorial problems, e.g. with subgraph counting  \cite{JaRuInf, JanOleRu,ChatterjeeTr,DeMarcoKahnTr,DeMarcoKahnCl}.  The best known result for general polynomial in this area has been obtained by Kim and Vu \cite{KimVuConc,VuConc}, who presented a family of powerful inequalities for $[0,1]$-valued random variables. Over the last decade they have been applied successfully to handle many problems in probabilistic combinatorics. Some recent inequalities for polynomials in the so called  subexponential random variables have been also obtained by Schudy and Sviridenko \cite{SchudySviridenko1,SchudySviridenko2}. They are a generalization of the special case of exponential random variables in \cite{LL} and are expressed in terms of quantities similar to those considered by Kim-Vu.

Since it is beyond the scope of this paper to give a precise account of all the concentration inequalities for polynomials, we refer the reader to the aforementioned sources and recommend also the monographs \cite{KwapienWoyczynski,dlPg}, where some parts of the theory are presented in a uniform way. As already mentioned we will present in detail only the results from \cite{L2}, which are our main tool as well as motivation.

As for concentration results for general non-Lipschitz functions, the only reference we are aware of, which addresses this question is \cite{GuillinJoulin}, where the Authors obtain interesting inequalities for stationary measures of certain Markov processes and functions satisfying a Lyapunov type condition. Their bounds are not comparable to the ones which we present in this paper. On the one hand they work in a more general Markov process setting, on the other hand, when specialized, e.g., to quadratic forms of Gaussian vectors, they do not recover optimal inequalities given in \cite{Bo,AG,L2} (see Section 4 in \cite{GuillinJoulin}). Since the language of \cite{GuillinJoulin} is very different from ours, we will not describe the inequalities obtained therein and refer the interested reader to the original paper.

Let us now proceed to the presentation of our results. To do this we will first formulate a two sided tail and moment inequality for homogeneous tetrahedral polynomials in i.i.d. standard Gaussian variables due to Lata{\l}a \cite{L2}. To present it in a concise way we need to introduce some notation which we will use throughout the article. For a positive integer $n$ we will denote $[n] = \{1,\ldots,n\}$. The cardinality of a set $I$ will be denoted by $\# I$. For $\ii = (i_1,\ldots,i_d) \in [n]^d$ and $I\subseteq[d]$ we write $\ii_{I}=(i_{k})_{k\in I}$. We will also denote $|\ii| = \max_{j\le d} {i_j}$.

Consider thus a $d$-indexed matrix $A = (a_{i_1,\ldots,i_d})_{i_1,\ldots,i_d = 1}^n$, such that $a_{i_1,\ldots,i_d} = 0$ whenever $i_j = i_k$ for some $j\neq k$, a sequence $g_1,\ldots,g_n$
of i.i.d. $\mathcal{N}(0,1)$ random variables and define
\begin{align}\label{eq:Z_def}
Z = \sum_{\ii \in [n]^d} a_{\ii} g_{i_1}\cdots g_{i_d}.
\end{align}
Without loss of generality we can assume that the matrix $A$ is symmetric, i.e., for all permutations $\sigma\colon [n]\to [n]$, $a_{i_1,\ldots,i_d} = a_{\sigma(i_1),\ldots,\sigma(i_d)}$.

Let now $P_d$ be the set of partitions of $\{1,\ldots,d\}$ into nonempty, pairwise disjoint sets. For a partition $\mathcal{J} =\{J_1,\ldots,J_k\}$, and a $d$-indexed matrix $A = (a_\ii)_{\ii \in [n]^d}$ (non-necessarily symmetric or with zeros on the diagonal), define
\begin{equation}\label{Gaussian_norm_def_intro}
\|A\|_{{\cal J}}=\sup\Big\{\sum_{\ii\in [n]^d} a_{\ii}\prod_{l=1}^k x\ub{l}_{\bfi_{J_l}}\colon
\|(x\ub{l}_{\bfi_{J_l}})\|_2\leq 1, 1\leq l\leq k \Big\},
\end{equation}
where $\|(x_{\ii_{J_l}})\|_2 = \sqrt{\sum_{|\ii_{J_l}|\le n} x_{\ii_{J_l}}^2}$. Thus, e.g.,
\begin{align*}
\|(a_{ij})_{i,j\le n}\|_{\{1,2\}}&= \sup\{ \sum_{i,j\le n} a_{ij}x_{ij}\colon \sum_{i,j\le n} x_{ij}^2 \le 1\} = \sqrt{\sum_{i,j\le n}a_{ij}^2} = \|(a_{ij})_{i,j\le n}\|_{\textup{HS}},\\
\|(a_{ij})_{i,j\le n}\|_{\{1\}\{2\}}&= \sup\{ \sum_{i,j\le n} a_{ij}x_iy_j\colon \sum_{i\le n} x_{i}^2\le 1,\sum_{j\le n}y_j^2 \le 1\} = \|(a_{ij})_{i,j\le n}\|_{\ell_2^n\to \ell_2^n},\\
\|(a_{ijk})_{i,j,k\le n}\|_{\{1,2\} \{3\}} &= \sup\{ \sum_{i,j,k\le n} a_{ij}x_{ij}y_k\colon \sum_{i,j\le n} x_{ij}^2\le 1,\sum_{k\le n}y_k^2 \le 1\}.
\end{align*}

From the functional analytic perspective the above norms are injective tensor product norms of $A$ seen as a multilinear form on $ (\R^{n})^d$ with the standard Euclidean structure.

We are now ready to present the inequalities by Lata{\l}a. Below, as in the whole article by $C_d$ we denote a constant, which depends only on $d$. The values of $C_d$ may differ between occurrences.

\begin{theorem}\label{thm:Latala_intro} For any $d$-indexed symmetric matrix $A = (a_{\ii})_{\ii \in [n]^d}$ such that $a_\ii = 0$ if $i_j = i_k$ for some $j\neq k$, the random variable $Z$, defined by \eqref{eq:Z_def} satisfies for all $p \ge 2$,
\begin{displaymath}
C_d^{-1} \sum_{\mathcal{J}\in P_d} p^{\#\mathcal{J}/2} \|A\|_{\mathcal{J}} \le \|Z\|_p \le C_d \sum_{\mathcal{J}\in P_d} p^{\#\mathcal{J}/2} \|A\|_{\mathcal{J}}.
\end{displaymath}
As a consequence, for all $t > 1$,
\begin{displaymath}
C_d^{-1}\exp\Big(-C_d\min_{\mathcal{J}\in P_d} \Big(\frac{t}{\|A\|_\mathcal{J}}\Big)^{2/\#\mathcal{J}}\Big)
 \le \p(|Z| \ge  t) \le C_d\exp\Big(-\frac{1}{C_d}\min_{\mathcal{J}\in P_d} \Big(\frac{t}{\|A\|_\mathcal{J}}\Big)^{2/\#\mathcal{J}}\Big).
\end{displaymath}
\end{theorem}

It is worthwhile noting that for $\#\mathcal{J} > 1$, the norms $\|A\|_{\mathcal{J}}$ are not unconditional in the standard basis (decreasing coefficients of the matrix may not result in decreasing the norm). Moreover, for specific matrices they may not be easy to compute. On the other hand, for any $d$-indexed matrix $A$ and any
$\mathcal{J} \in P_d$, we have $\|A\|_\mathcal{J} \le \|A\|_{\{1,\ldots,d\}} = \sqrt{\sum_{\ii} a_\ii^2}$. Using this fact in the upper estimates above allows to recover (up to constants depending on $d$) hypercontractive estimates for homogeneous tetrahedral polynomials due to Nelson.

Our main result is an extension of the upper bound given in the above theorem to more general random functions and measures. Below we present the most basic setting we will work with and state the corresponding theorems. Some additional extensions are deferred to the main body of the article.

We will consider a random vector $X$ in $\R^n$, which satisfies the following family of Sobolev inequalities. For any $p \ge 2$ and any smooth integrable function $f\colon \R^n \to \R$,
\begin{align}\label{eq:sobolev_def}
\|f(X)-\E f(X)\|_p \le L\sqrt{p}\Big\||\nabla f(X)|\Big\|_p,
\end{align}
for some constant $L$ (independent of $p$ and $f$), where $|\cdot|$ is the standard Euclidean norm on $\R^n$. It is known (see \cite{Aida-Stroock-1994} and Theorem \ref{thm:AidaStroock} below) that if $X$ satisfies the logarithmic Sobolev inequality with constant $D_{LS}$, then it satisfies \eqref{eq:sobolev_def} with $L = \sqrt{D_{LS}/2}$. We remark that there are many criteria for a random vector to satisfy the logarithmic Sobolev inequality (see e.g. \cite{LedouxConcBook,BakryEmery,BobkovGoetze,BartheMilmanTransference,KlartagCube}), so in particular our assumption \eqref{eq:sobolev_def} can be verified for many random vectors of interest.

Our first result is the following theorem, which provides moment estimates and concentration for $D$-times differentiable functions. The estimates are expressed by $\|\cdot\|_{\mathcal{J}}$ norms of derivatives of the function (which we will identify with multi-indexed matrices). We will denote the $d$-th derivative of $f$ by $\D^d f$.

\begin{theorem}\label{thm:main_intro} Assume that a random vector $X$ in $\R^n$ satisfies the inequality \eqref{eq:sobolev_def} with constant $L$. Let $f\colon \R^n \to \R$ be a function of the class $\mathcal{C}^D$. For all $p \ge 2$ if $\D^Df(X) \in L^p$, then
\begin{displaymath}
\|f(X) - \E f(X)\|_p \le C_D\Big(L^D\sum_{\mathcal{J} \in P_D} p^{\frac{\#\mathcal{J}}{2}} \Big\|\|\D^Df(X)\|_\mathcal{J}\Big\|_p
+ \sum_{1\le d\le D-1}L^d\sum_{\mathcal{J}\in P_d} p^{\frac{\#\mathcal{J}}{2}} \|\E \D^df(X)\|_\mathcal{J}\Big).
\end{displaymath}
In particular if $\D^Df (x)$ is uniformly bounded on $\R^n$, then setting
\begin{displaymath}
\eta_f(t) = \min\left(\min_{\mathcal{J}\in P_D}\Big(\frac{t}{L^{D}\sup_{x\in \R^n}\|\D^D f(x)\|_\mathcal{J}}\Big)^{\frac{2}{\#\mathcal{J}}},\min_{1\le d\le D-1}\min_{\mathcal{J}\in P_d} \Big(\frac{t}{L^{d}\|\E \D^d f(X)\|_\mathcal{J}}\Big)^{\frac{2}{\#\mathcal{J}}}\right)
\end{displaymath}
we obtain for $t > 0$,
\begin{displaymath}
\p(|f(X) - \E f(X)| \ge t ) \le 2\exp\Big(-\frac{1}{C_D} \eta_f(t)\Big).
\end{displaymath}
\end{theorem}

The above theorem is quite technical, so we will now provide a few comments, comparing it to known results.
\paragraph{1.} It is easy to see that if $D = 1$, Theorem \ref{thm:main_intro} reduces (up to absolute constants) to the Gaussian-like concentration inequality, which can be obtained from \eqref{eq:sobolev_def} by Chebyshev's inequality (applied to general $p$ and optimized).

\paragraph{2.} If $f$ is a homogeneous tetrahedral polynomial of degree $D$, then the tail and moment estimates of Theorem \ref{thm:main_intro} coincide with those from Lata{\l}a's Theorem. Thus Theorem \ref{thm:main_intro} provides an extension of the upper bound from Lata{\l}a 's result to a larger class of measures and functions (however we would like to stress that our proof relies heavily on Lata{\l}a's work).

\paragraph{3.} If $f$ is a general polynomial of degree $D$, then $\D^D f(x)$ is constant on $\R^n$ (and thus equal to $\E \D^D f(X)$). Therefore in this case the function $\eta_f$ appearing in Theorem \ref{thm:main_intro} can be written in a simplified form
\begin{align}\label{eq:eta_def_poly}
\eta_f(t) = \min_{1\le d\le D}\min_{\mathcal{J}\in P_d} \Big(\frac{t}{L^{d}\|\E \D^d f(X)\|_\mathcal{J}}\Big)^{2/\#\mathcal{J}}.
\end{align}

\paragraph{4.} For polynomials in Gaussian variables, the estimates given in Theorem \ref{thm:main_intro} can be reversed, like in Theorem \ref{thm:Latala_intro}. More precisely we have the following theorem, which provides an extension of Theorem \ref{thm:Latala_intro} to general polynomials.
\begin{theorem}\label{thm:Gaussian_intro} If $G$ is a standard Gaussian vector in $\R^n$ and $f\colon \R^n \to \R$ is a polynomial of degree $D$, then for all $p \ge 2$,
\begin{displaymath}
C_D^{-1} \sum_{1\le d\le D}\sum_{\mathcal{J}\in P_d} p^{\frac{\#\mathcal{J}}{2}} \|\E \D^df(G)\|_\mathcal{J}
\le \|f(G) - \E f(G)\|_p
\le C_D \sum_{1\le d\le D}\sum_{\mathcal{J}\in P_d} p^{\frac{\#\mathcal{J}}{2}} \|\E \D^df(G)\|_\mathcal{J}.
\end{displaymath}
Moreover for all $t > 0$,
\begin{displaymath}
\frac{1}{C_D}\exp\Big(-C_D \eta_f(t)\Big) \le \p(|f(G) - \E f(G)| \ge t) \le C_D\exp\Big(-\frac{1}{C_D} \eta_f(t)\Big),
\end{displaymath}
where
\begin{align*}
\eta_f(t) = \min_{1\le d\le D}\min_{\mathcal{J}\in P_d} \Big(\frac{t}{\|\E \D^d f(G)\|_\mathcal{J}}\Big)^{2/\#\mathcal{J}}.
\end{align*}
\end{theorem}

\paragraph{5.} It is well known that concentration of measure for general Lipschitz functions fails e.g. on the discrete cube and one has to impose some additional convexity assumptions to get sub-Gaussian concentration \cite{TalCube}. It turns out that if we restrict to polynomials, estimates in the spirit of Theorems \ref{thm:Latala_intro} and \ref{thm:main_intro} still hold. To formulate our result in full generality recall the definition of the $\psi_2$ Orlicz norm of a random variable $Y$,
\begin{displaymath}
\|Y\|_{\psi_2} = \inf\Big\{ t > 0\colon \E \exp\Big(\frac{Y^2}{t^2}\Big) \le 2\Big\}.
\end{displaymath}
By integration by parts and Chebyshev's inequality $\|Y\|_{\psi_2} < \infty$ is equivalent to a sub-Gaussian tail decay for $Y$. We have the following result for polynomials in sub-Gaussian random vectors with independent components.
\begin{theorem}\label{thm:subgaussian_intro}
Let $X = (X_1,\ldots,X_n)$ be a random vector with independent components, such that for all $i \le n$, $\|X_i\|_{\psi_2} \le L$. Then for every polynomial $f\colon \R^n \to \R$ of degree $D$ and every $p \ge 2$,
\begin{displaymath}
\|f(X) - \E f(X)\|_p \le C_D \sum_{d=1}^D L^d \sum_{\mathcal{J}\in \mathcal{P}_d} p^{\#\mathcal{J}/2}\|\E \D^d f(X)\|_\mathcal{J}.
\end{displaymath}
As a consequence, for any $t > 0$,
\begin{displaymath}
\p\Big(|f(X) - \E f(X)| \ge  t\Big) \le 2\exp\Big(-\frac{1}{C_D}\eta_f(t)\Big),
\end{displaymath}
where
\begin{align*}
\eta_f(t) = \min_{1\le d\le D}\min_{\mathcal{J}\in P_d} \Big(\frac{t}{L^{d}\|\E \D^d f(X)\|_\mathcal{J}}\Big)^{2/\#\mathcal{J}}.
\end{align*}
\end{theorem}

\paragraph{6.} We postpone the applications of our theorems to subsequent sections of the article and here we announce only that apart from polynomials we apply Theorem \ref{thm:main_intro} to additive functionals and $U$-statistics of random vectors, in particular to linear eigenvalue statistics of random matrices, obtaining bounds which complement known estimates by Guionnet and Zeitouni \cite{GZConc}. Theorem \ref{thm:subgaussian_intro} is applied to the special case of the problem of subgraph counting in large random graphs. In a special case when one counts copies of a given small cycle, our result allows to obtain optimal inequalities for random graphs $G(n,p)$, with $p \to 0$ slowly, namely $p \ge n^{-\frac{k-2}{2(k-1)}} \log^{-\frac12} n$, where $k$ is length of a cycle. To the best of our knowledge they are the best currently known inequalities for this range of $p$.

\paragraph{7.} Let us now briefly discuss optimality of our inequalities. The lower bound in Theorem \ref{thm:Gaussian_intro} clearly shows that Theorem \ref{thm:main_intro} is optimal in the class of measures and functions it covers up to constants depending only on $D$. As for Theorem \ref{thm:subgaussian_intro}, it is similarly optimal in the class of random vectors with independent sub-Gaussian coordinates. In concrete combinatorial applications, for $0$-$1$ random variables this theorem may be however suboptimal. This can be seen already for $D = 1$, for a linear combination of independent Bernoulli variables $X_1,\ldots,X_n$ with $\p(X_i = 1) = 1 - \p(X_i=0) = p$. When $p$ becomes small, the tail bound for such variables given e.g. by the Chernoff inequality is more subtle than what can be obtained from general inequalities for sums of sub-Gaussian random variables and the fact that $\|X_i\|_{\psi_2}$ is of order $(\log(2/p))^{-1/2}$. Roughly speaking, this is the reason why in our estimates for random graphs we have a restriction on the speed at which $p \to 0$. At the same time our inequalities still give results comparable to what can be obtained from other general inequalities for polynomials. As already noted in the survey \cite{JaRuInf}, bounds obtained from various general inequalities for the subgraph-counting problem, may not be directly comparable, i.e. those performing well in one case may exhibit worse performance in some other cases. Similarly, our inequalities cannot be in general compared e.g. to the estimates by Kim and Vu. For this reason and since it would require introducing new notation, we will not discuss these inequalities and just indicate, when presenting applications of Theorem \ref{thm:subgaussian_intro}, several situations when our inequalities perform in a better or worse way than those by Kim and Vu. Let us only mention that the Kim-Vu inequalities similarly as ours are expressed in terms of higher order derivatives of the polynomials. However, Kim and Vu (as well as Schudy and Sviridenko) look at maxima of absolute values of partial derivatives, which does  not lead to tensor-product norms which we consider. While in the general sub-Gaussian case we consider, such tensor product norms cannot be avoided (in view of Theorem \ref{thm:Gaussian_intro}), it is not necessarily the case for $0$-$1$ random variables.

\bigskip
The organization of the paper is as follows. First, in Section \ref{sec:Notation}, we introduce the notation used in the paper, next in Section \ref{sec:general_nonlipschitz} we give the proof of Theorem \ref{thm:main_intro} together with some generalizations and examples of applications. In Section \ref{sec:Gaussian} we prove Theorem \ref{thm:Gaussian_intro}, whereas in Section \ref{sec:subgaussian} we present the proof of Theorem \ref{thm:subgaussian_intro} and applications to the subgraph counting problems. In Section \ref{sec:Weibull} we provide further refinements of estimates from Section \ref{sec:general_nonlipschitz} in the case of independent random variables satisfying modified log-Sobolev inequalities (they are deferred to the end of the article as they are more technical than those of Section \ref{sec:general_nonlipschitz}). In the Appendix we collect some additional facts used in the proofs.

\paragraph{Acknowledgement} We would like to thank Michel Ledoux and Sandrine Dallaporta for interesting discussions concerning tail estimates for linear eigenvalue statistics of random matrices.

\section{Notation}\label{sec:Notation}

\paragraph{Sets and indices}
For a positive integer $n$ we will denote $[n] = \{1,\ldots,n\}$. The cardinality of a set $I$ will be denoted by $\# I$.

For $\ii = (i_1,\ldots,i_d)\in [n]^d$ and $I\subseteq[d]$ we write $\ii_{I}=(i_{k})_{k\in I}$. We will also denote $|\ii| = \max_{j\le d} {i_j}$.

For a finite set $A$ and an integer $d \ge 0$ we set
\[
  A\uu{d} = \{\ii = (i_1,\ldots,i_d) \in A^d\colon \forall_{j,k \in \{1,\ldots,d\}} \ j\neq k \Rightarrow i_j\neq i_k\}
\]
(i.e. $A\uu{d}$ is the set of $d$-indices with pairwise distinct coordinates). Accordingly we will denote $n\uu{d} = n(n-1)\cdots(n-d+1)$.

By $P_d$ we will denote the family of partitions of $[d]$ into nonempty, pairwise disjoint sets.

For a finite set $I$ by $\ell_2(I)$ we will denote the finite dimensional Euclidean space $\R^I$ endowed with the standard Euclidean norm $|x|_2 = \sqrt{\sum_{i\in I} x_i^2}$. Whenever there is no risk of confusion we will denote the standard Euclidean norm simply by $|\cdot|$.

\paragraph{Multi-indexed matrices}
For a function $f\colon \R^n \to \R$ by $\D^d f(x)$ we will denote the ($d$-indexed) matrix of its derivatives of order $d$, which we will identify with the corresponding symmetric $d$-linear form. If $M = (M_\ii)_{\ii\in [n]^d}$, $N = (N_\ii)_{\ii\in [n]^d}$ are $d$-indexed matrices, we define $\langle M,N\rangle =\sum_{\ii\in [n]^d} M_\ii N_\ii$. Thus for all vectors $y_1,\ldots,y_d \in \R^n$ we have $\D^d f(x) (y_1,\ldots,y_d) = \langle \D^d f(x),y_1\otimes\cdots\otimes y_d\rangle$, where $y_1\otimes\cdots\otimes y_d = (y_{i_1}y_{i_2}\cdots y_{i_d})_{\ii \in [n]^d}$.

We will also define the Hadamard product of two such matrices $M\circ N$ as a $d$-indexed matrix with entries $m_{\ii} = M_\ii N_\ii$ (pointwise multiplication of entries).

Let us also define the notion of ``generalized diagonals'' of a $d$-indexed matrix $A = (a_\ii)_{\ii \in [n]^d}$. For a fixed set $K \subseteq [d]$, with $\#K > 1$, the ``generalized diagonal'' corresponding to $K$  is is the set of indices $\{\ii \in [n]^d\colon i_k = i_l\;\textrm{for}\;  k,l \in K\}$.

\paragraph{Constants} We will use the letter $C$ to denote absolute constants and $C_a$ for constants depending only on some parameter $a$. In both cases the values of such constants may differ between occurrences.

\section{A concentration inequality for non-lipschitz functions}\label{sec:general_nonlipschitz}
In this Section we prove Theorem \ref{thm:main_intro}. Let us first state our main tool, which is an inequality by Lata{\l}a in a decoupled version.

\begin{theorem}[Lata{\l}a, \cite{L2}]\label{thm_Latala_dec} Let $A = (a_\ii)_{\ii \in [n]^d}$ be a $d$-indexed matrix with real entries and let $G_1,G_2,\ldots,G_d$ be i.i.d. standard Gaussian vectors in $\R^n$. Let $Z = \langle A, G_1\otimes\cdots\otimes G_d\rangle$. Then for every $p\ge 2$,
\begin{displaymath}
C_d^{-1}\sum_{\mathcal{J}\in P_d}p^{\#\mathcal{J}/2}\|A\|_\mathcal{J} \le \|Z\|_p \le C_d\sum_{\mathcal{J}\in P_d}p^{\#\mathcal{J}/2}\|A\|_\mathcal{J}
\end{displaymath}
\end{theorem}

Thanks to general decoupling inequalities for $U$-statistics \cite{dlPMS1}, which we recall in the Appendix (Theorem \ref{thm:decoupling}), the above theorem is formally equivalent to Theorem \ref{thm:Latala_intro}. In fact in \cite{L2} Lata{\l}a first proves the above version. In the proof of Theorem \ref{thm:main} we will need just Theorem \ref{thm_Latala_dec} (in particular in this part of the article we do not need any decoupling inequalities).

From now on we will work in a more general setting than in Theorem \ref{thm:main_intro} and assume that $X$ is a random vector in $\R^n$, such that for all $p\ge 2$ there exists a constant $L_X(p)$ such that for all bounded $\mathcal{C}^1$ functions $f \colon \R^n \to \R$,
\begin{equation}\label{eq_main_assumption_2}
\|f(X) - \E f(X)\|_p \le L_X(p) \Big\||\nabla f(X)|\Big\|_p.
\end{equation}
Clearly in this situation the above inequality generalizes to all $\mathcal{C}^1$ functions (if the right-hand side is finite then the left-hand side is well defined and the inequality holds).

Let now $G$ be a standard $n$-dimensional Gaussian vector, independent of $X$. Using the Fubini theorem together with the fact that for some absolute constant $C$, all $x \in \R^n$ and $p \ge 2$,
$C^{-1}\sqrt{p}|x| \le \|\langle x, G\rangle\|_p \le C\sqrt{p}|x|$, we can linearise the right-hand side above and write \eqref{eq_main_assumption_2} equivalently (up to absolute constants) as
\begin{align}\label{eq:linearization}
\|f(X) - \E f(X)\|_p \le \frac{C L_X(p)}{\sqrt{p}} \Big\|\langle \nabla f(X),G\rangle\Big\|_p.
\end{align}

We remark that similar linearisation has been used by Maurey and Pisier to provide a simple proof of the Gaussian concentration inequality \cite{PisierProbabMethods,PisierVolume} (see remark following Theorem \ref{thm:main} below).
Inequality \eqref{eq:linearization} has an advantage over \eqref{eq_main_assumption_2} as it allows for iteration leading to the following simple proposition.

\begin{prop}\label{prop:moment_Poincare}
Consider $p \ge 2$ and let $X$ be an $n$-dimensional random vector satisfying \eqref{eq_main_assumption_2}.
Let $f \colon \R^n \to \R$ be a $\mathcal{C}^D$ function. Let moreover $G_1,\ldots,G_D$ be independent standard Gaussian vectors in $\R^n$, independent of $X$.
Then for all $p \ge 2$, if $\D^D f(X) \in L^p$, then
\begin{align}\label{eq:moment_estimate}
\|f(X) - \E f(X)\|_p  \le&  \frac{C^{D}L_X(p)^D}{p^{D/2}}\|\langle \D^D f(X), G_1\otimes\cdots\otimes G_D\rangle\|_p \\
&+ \sum_{1\le d\le D-1} \frac{C^{d}L_X(p)^d}{p^{d/2}}\|\langle \E_X \D^d f(X), G_1\otimes\cdots\otimes G_d\rangle\|_p\nonumber.
\end{align}

\end{prop}

\begin{proof} Induction on $D$. For $D = 1$ the assertion of the proposition coincides with \eqref{eq:linearization}, which (as already noted) is equivalent to \eqref{eq_main_assumption_2}. Let us assume that the proposition holds for $D-1$. Applying thus \eqref{eq:moment_estimate} with $D-1$ instead of $D$, we obtain
\begin{align}\label{eq:main_prop_aux_1}
\|f(X) - \E f(X)\|_p \le& \frac{C^{D-1}L_X(p)^{D-1}}{p^{(D-1)/2}}\|\langle \D^{D-1} f(X),G_1\otimes\cdots\otimes G_{D-1}\rangle\|_p\\
&+ \sum_{d=1}^{D-2} \frac{C^{d}L_X(p)^{d}}{p^{d/2}}\|\langle \E_X \D^d f(X),G_1\otimes\cdots\otimes G_d\rangle\|_p.\nonumber
\end{align}
Applying now the triangle inequality in $L^p$, we get
\begin{align}
\|\langle \D^{D-1} f(X),G_1\otimes\cdots\otimes G_{D-1}\rangle\|_p \le& \|\langle \D^{D-1} f(X) - \E_X \D^{D-1}f(X),G_1\otimes\cdots\otimes G_{D-1}\rangle\|_p \nonumber\\
&+ \|\langle \E_X \D^{D-1}f(X),G_1\otimes\cdots\otimes G_{D-1}\rangle\|_p.\label{eq:main_prop_aux_2}
\end{align}
Let us now apply \eqref{eq:linearization} conditionally on $G_1,\ldots,G_{D-1}$ to the function $f_1(x) = \langle \D^{D-1} f(x),G_1\otimes\cdots\otimes G_{D-1}\rangle$. Since $\langle \D^{D-1} f(X) - \E_X \D^{D-1}f(X),G_1\otimes\cdots\otimes G_{D-1}\rangle = f_1(X) - \E_X f_1(X)$) and
$\langle \nabla f_1(X),G_D\rangle= \langle \D^D f(X),G_1\otimes\cdots\otimes G_D\rangle$, we obtain
\begin{align*}
&\E_X |\langle \D^{D-1} f(X)- \E_X\D^{D-1} f(X),G_1\otimes\cdots\otimes G_{D-1}\rangle|^p \\
&\le \frac{C^pL_X(p)^p}{p^{p/2}} \E_{X,G_D} |\langle \D^D f(X),G_1\otimes\cdots\otimes G_D\rangle|^p.
\end{align*}
To finish the proof it is now enough to integrate this inequality with respect to the remaining Gaussian vectors and combine the obtained estimate with
\eqref{eq:main_prop_aux_1} and \eqref{eq:main_prop_aux_2}.
\end{proof}

Let us now specialize to the case when $L_X(p) = Lp^\gamma$ for some $L>0,\gamma \ge 1/2$. Combining the above proposition with Lata{\l}a's Theorem \ref{thm_Latala_dec}, we obtain immediately the following theorem, a special case of which is Theorem \ref{thm:main_intro}.

\begin{theorem}\label{thm:main}
Assume that $X$ is a random vector in $\R^n$, such that for some constants $L>0,\gamma\ge 1/2$,  all smooth functions $f$ and all $p \ge 2$,
\begin{align}\label{eq:Sobolev_gamma}
\|f(X)-\E f(X)\|_p \le Lp^\gamma\Big\||\nabla f(X)|\Big\|_p.
\end{align}
For any smooth function $f\colon \R^n \to \R$ of class $\mathcal{C}^D$ and $p \ge 2$ if $\D^D f(X) \in L^p$, then
\begin{align*}
\|f(X)-\E f(X)\|_p \le& C_D\Big(\sum_{\mathcal{J}\in P_D} L^Dp^{(\gamma - 1/2)D + \#\mathcal{J}/2} \Big\|\|\D^D f(X)\|_\mathcal{J}\Big\|_p\\
&+\sum_{1\le d\le D-1} \sum_{\mathcal{J}\in P_d} L^d p^{(\gamma-1/2)d + \# \mathcal{J}/2}\|\E \D^d f(X)\|_{\mathcal{J}}\Big).
\end{align*}
If $\D^D f $ is bounded uniformly on $\R^n$, then for all $t > 0$,
\begin{displaymath}
\p(|f(X)-\E f(X)| \ge t) \le 2\exp\Big(-\frac{1}{C_D}\eta_f(t)\Big),
\end{displaymath}
where
\begin{align*}
\eta_f(t) &= \min(A,B),\\
A &= \min_{\mathcal{J}\in P_D}\Big(\Big(\frac{t}{L^{D}\sup_{x\in \R^n}\|\D^D f(x)\|_\mathcal{J}}\Big)^{2/((2\gamma-1) D+ \#\mathcal{J})}\Big),\\
B &= \min_{1\le d\le D-1} \min_{\mathcal{J}\in P_d} \Big(\Big(\frac{t}{L^{d}\|\E \D^d f(X)\|_\mathcal{J}}\Big)^{2/((2\gamma -1)d+\#\mathcal{J})}\Big).
\end{align*}
\end{theorem}

\begin{proof}
The first part is a straightforward combination of Proposition \ref{prop:moment_Poincare} and Theorem \ref{thm_Latala_dec}. The second part follows from the first one by Chebyshev's inequality $\p(|Y|\ge e\|Y\|_p) \le \exp(-p)$ applied with $p = \eta_f(t)/C_D$ (note that if $\eta_f(t)/C_D\le 2$ then one can make the tail bound asserted in the theorem trivial by adjusting the constants).
\end{proof}

\paragraph{Remark} In  \cite{PisierProbabMethods,PisierVolume} Pisier presents a stronger inequality than \eqref{eq:Sobolev_gamma} with $\gamma = 1/2$. More specifically, he proves that if $X,G$ are independent standard centred Gaussian vectors in $\R^n$,  $E$ is a Banach space and $f \colon \R^n \to E$ is a $\mathcal{C}^1$ function, then for every convex function $\Phi \colon E \to \R$,
\begin{align}\label{eq_Pisier}
\E \Phi(f(X) - \E f(X)) \le \E \Phi\Big(L \langle \nabla f(X),G\rangle\Big),
\end{align}
where $L = \frac{\pi}{2}$. As noted in \cite{LedOle}, Caffarelli's contraction principle \cite{Caffarelli_contraction} implies that, e.g., a random vector $X$ with density $e^{-V}$, where $V\colon \R^n \to \R$ satisfies $D^2 V \ge \lambda \id$, $\lambda > 0$ satisfies the above inequality with $L = \frac{\pi}{2\sqrt{\lambda}}$ (where $G$ is still a standard Gaussian vector independent of $X$).
Therefore in this situation a similar approach as in the proof of Proposition \ref{prop:moment_Poincare} can be used for functions $f$ with values in a general Banach space. Moreover, a counterpart of Lata{\l}a's results is known for chaoses with values in a Hilbert space (to the best of our knowledge this observation has not been published, in fact it can be quite easily obtained from the version for real valued chaoses). Thus in this case we can obtain a counterpart of Theorem \ref{thm:main} (with $\gamma = 1/2$) for Hilbert space valued-functions. In the case of a general Banach space two-sided estimates for Banach space-valued Gaussian chaoses are not known. Still, one can use some known inequalities (like hypercontraction or Borell-Arcones-Gin\'{e} inequality) instead of Theorem \ref{thm_Latala_dec} and thus obtain new concentration bounds. We remark that if one uses hypercontraction, one can obtain explicit dependence of the constants on the degree of the polynomial, since explicit constants are known for hypercontractive estimates of (Banach space-valued) Gaussian chaoses and one can keep track of them during the proof. We skip the details.

\bigskip
In view of Theorem \ref{thm:main} a natural question arises: for what measures is the inequality \eqref{eq:Sobolev_gamma} satisfied?
Before we provide examples, for technical reasons let us recall the definition of the length of the gradient of a locally Lipschitz function. For a metric space $(\mathcal{X},d)$, a locally Lipschitz function $f\colon \mathcal{X} \to \R$ and $x \in \mathcal{X}$, we define
\begin{align}\label{eq:length_of_gradient}
|\nabla f|(x) = \limsup_{d(x,y)\to 0} \frac{|f(y)-f(x)|}{d(x,y)}.
\end{align}

If $\mathcal{X} = \R^n$ and $f$ is differentiable at $x$, then clearly $|\nabla f|(x)$ coincides with the Euclidean length of the usual gradient $\nabla f(x)$. For this reason, with slight abuse of notation, we will write $|\nabla f(x)|$ instead of $|\nabla f|(x)$. We will consider only measures on $\R^n$, however since we allow measures which are not necessarily absolutely continuous with respect to the Lebesgue measure, at some points in the proofs we will work with the above abstract definition.

Going back to the question of measures satisfying \eqref{eq:Sobolev_gamma}, it is well known (see e.g. \cite{Milman_role_iso}) that if $X$ satisfies the Poincar\'e inequality
\begin{align}\label{eq:Poincare}
\Var(f(X)) \le D_{Poin}\E|\nabla f(X)|^2
\end{align}
for all locally Lipschitz bounded functions, then $X$ satisfies \eqref{eq:Sobolev_gamma} with $\gamma = 1$ and $L = C\sqrt{D_{Poin}}$ (recall that $C$ always denotes a universal constant).
Assume now that $X$ satisfies the logarithmic Sobolev inequality
\begin{equation}\label{eq_log_Sobolev}
 \Ent f^2(X) \le D_{LS} \E |\nabla f(X)|^2
\end{equation}
for locally Lipschitz bounded functions, where for a nonnegative random variable $Y$,
\begin{displaymath}
\Ent Y = \E Y\log Y - \E Y\log(\E Y).
\end{displaymath}
Then, by the results from \cite{Aida-Stroock-1994}, it follows that $X$ satisfies \eqref{eq:Sobolev_gamma} with $\gamma = 1/2$ and $L = \sqrt{D_{LS}/2}$.

We will now generalize this observation to measures satisfying the so-called modified logarithmic Sobolev inequality (introduced in \cite{Gentil-Guillin-Miclo-2005}). We will present it in greater generality than needed for proving \eqref{eq:Sobolev_gamma}, since we will use it later (in Section \ref{sec:Weibull}) to prove refined concentration results for random vectors with independent Weibull coordinates.

Let $\beta \in (2,\infty)$. We will say that a random vector $Y \in \R^k$ satisfies a $\beta$-modified logarithmic Sobolev inequality if for every locally Lipschitz bounded positive function $f \colon \R^k \to \R$,
\begin{align}\label{eq:modifiedLS}
\Ent f^2(Y) \le
   D_{LS_\beta} \Big(\E |\nabla f(Y)|^2 + \E \frac{|\nabla f(Y)|^\beta}{f(Y)^{\beta-2}}\Big).
\end{align}

Let us also introduce two quantities, measuring the length of the gradient in product spaces.  Consider a locally Lipschitz function $f \colon \R^{mk} \to \R$, where we identify $R^{mk}$ with the $m$-fold Cartesian product of $\R^k$. Let $x = (x_1,\ldots,x_m)$, where $x_i \in \R^k$. For each $i=1,\ldots,m$, let $|\nabla_i f(x)|$ be the length of the gradient of $f$, treated as a function of $x_i$ only, with the other coordinates fixed. Now for $r \ge 1$, set
\begin{displaymath}
|\nabla f(x)|_r = \Big(\sum_{i=1}^m |\nabla_i f(x)|^r\Big)^{1/r}.
\end{displaymath}
Note that if $f$ is differentiable at $x$, then $|\nabla f(x)|_2 = |\nabla f(x)|$ (the Euclidean length of the ``true'' gradient), whereas for $k = 1$ (and $f$ differentiable), $|\nabla f(x)|_r$ is the $\ell_r^m$ norm of $\nabla f(x)$.

\begin{theorem}\label{thm:AidaStroock} Let $\beta \in [2,\infty)$ and $Y$ be a random vector in $\R^k$, satisfying \eqref{eq:modifiedLS}. Consider a random vector $X = (X_1,\ldots,X_m)$ in $\R^{mk}$, where $X_1,\ldots,X_m$ are independent copies of $Y$. Then for any locally Lipschitz $f \colon \R^{mk} \to \R$ such that $f(X)$ is integrable, and $p \ge 2$,
\begin{equation}\label{ineq:sobolev-ineq-from-beta-modified}
\|f(X) - \E f(X)\|_p \le C_\beta D_{LS_\beta}^{1/2} p^{1/2}\Big\||\nabla f(X)|_2\Big\|_p + D_{LS_\beta}^{1/\beta}p^{1/\alpha}\Big\||\nabla f(X)|_\beta\Big\|_p,
\end{equation}
where $\alpha = \frac{\beta}{\beta-1}$ is the H\"older conjugate of $\beta$.
\end{theorem}

In particular using the above theorem with $m = 1$ and $k = n$, we obtain the following

\begin{cor}If $X$ is a random vector in $\R^n$ which satisfies the $\beta$-modified log-Sobolev inequality \eqref{eq:modifiedLS}, then it satisfies \eqref{eq:Sobolev_gamma} with $\gamma = \frac{\beta-1}{\beta} \ge \frac{1}{2}$ and $L = C_\beta \max(D_{LS_\beta}^{1/2},D_{LS_\beta}^{1/\beta})$.
\end{cor}

We remark that  in the class of logarithmically concave random vectors, the $\beta$-modified log-Sobolev inequality is known to be equivalent to concentration for 1-Lipschitz functions of the form $\p(|f(X) - \E f(X)| \ge t) \le 2\exp(-c t^{\beta/(\beta-1)})$ \cite{MilmanProperties}.

\begin{proof}[Proof of Theorem \ref{thm:AidaStroock}] By the tensorization property of entropy  (see e.g. \cite{LedouxConcBook}, Proposition 5.6) we get
for all positive locally Lipschitz bounded functions $f \colon \R^{mk} \to \R$,
\begin{align}\label{eq:modified_LS_tensorized}
\Ent f^2(X) \le D_{LS_\beta}\Big(\E |\nabla f(X)|_2^2 + \sum_{i=1}^m \E \frac{|\nabla_i f(X)|^\beta}{f(X)^{\beta -2}}\Big).
\end{align}

Following~\cite{Aida-Stroock-1994}, consider now any locally Lipschitz bounded $f > 0$ and denote $F(t) = \E f(X)^t$. For $t > 2$,
\[
  F'(t) = \E \left( f(X)^t \log f(X) \right)
\]
and
\[ \begin{split}
  \frac{d}{dt} \left( \E f(X)^t \right)^{2/t} &= \frac{d}{dt} F(t)^{2/t}
    = F(t)^{2/t} \cdot \frac{d}{dt} \left( \frac{2}{t} \log F(t) \right) \\[1ex]
&= F(t)^{2/t} \left( \frac{2}{t} \frac{F'(t)}{F(t)} - \frac{2}{t^2} \log F(t) \right)
  = \frac{2}{t^2} F(t)^{\frac{2}{t} - 1} \left( t F'(t) - F(t) \log F(t) \right) \\[1ex]
&= \frac{2}{t^2} \left( \E f(X)^t \right)^{\frac{2}{t} - 1}
    \left( \E \left( f(X)^t \log f(X)^t \right) - \left(\E f(X)^t \right) \log \left(\E f(X)^t \right) \right).
\end{split} \]
By  \eqref{eq:modified_LS_tensorized}  applied to the function $g = f^{t/2} = \varphi \circ f $ where $\varphi(u) = |u|^{t/2}$,
\begin{displaymath}
  \frac{d}{dt} \left( \E f(X)^t \right)^{2/t} \le
  \frac{2}{t^2} \left( \E f(X)^t \right)^{\frac{2}{t} - 1} \cdot
   D_{LS_\beta} \Big(\E |\nabla (\varphi \circ f)(X)|_2^2 + \E |\nabla (\varphi\circ f)(X)|_\beta^\beta f(X)^{t(2-\beta)/2}\Big).
\end{displaymath}
By the chain rule and the H{\"o}lder inequality for the pair of conjugate exponents $t/2, t/(t-2)$,
\[ \begin{split}
  \E \left|\nabla (\varphi \circ f)(X)\right|_2^2 &= \E \big( \left|\varphi'(f(X)) \right| \cdot \left| \nabla f(X)\right|_2 \big)^2 \\[1ex]
    &\le \left( \E |\nabla f(X)|_2^t \right)^{2/t}
           \left( \E \left(\varphi'(f(X))\right)^{2t/(t-2)} \right)^{(t-2)/t} \\[1ex]
    &= \big\||\nabla f(X)|_2\big\|_t^2 \cdot \left( \frac{t^2}{4} \right) \left(\E f(X)^t \right)^{1 - \frac{2}{t}}.
\end{split} \]
Similarly, for $t \ge \beta$,
\begin{align*}
\E |\nabla(\varphi\circ f)(X)|_\beta^\beta f(X)^{t(2-\beta)/2} & = \frac{t^\beta}{2^\beta}\E  f(X)^{(t/2-1)\beta}|\nabla f(X)|_\beta^\beta f(X)^{t(2-\beta)/2}\\
& = \frac{t^\beta}{2^\beta} \E f(X)^{t-\beta} |\nabla f(X)|_\beta^\beta\\
&\le \frac{t^\beta}{2^\beta} (\E f(X)^t)^{1-\beta/t} (\E |\nabla f(X)|_\beta^t)^{\beta/t}\\
&= \frac{t^\beta}{2^\beta} (\E f(X)^t)^{1-\beta/t} \big\| |\nabla f(X)|_\beta\big\|_t^\beta.
\end{align*}
Thus we get for $\beta \le t \le p$,
\begin{displaymath}
\frac{d}{dt} \left( \E f(X)^t \right)^{2/t} \le \frac{D_{LS_\beta}}{2}\big\||\nabla f(X)|_2\big\|_p^2 + \frac{D_{LS_\beta}}{2^{\beta-1}}t^{\beta -2}(\E f(X)^t)^{(2-\beta)/t}\big\||\nabla f(X)|_\beta\big\|_p^\beta.
\end{displaymath}
Denote $ a = \frac{D_{LS_\beta}}{2}\big\||\nabla f(X)|_2\big\|_p^2$, $b = \frac{D_{LS_\beta}}{2^{\beta-1}}\big\||\nabla f(X)|_\beta\big\|_p^\beta$, $g(t) = \left( \E f(X)^t \right)^{2/t}$. The above inequality can be written as
\begin{displaymath}
g^{\beta/2-1}\frac{d}{dt} g \le g^{\beta/2 - 1} a + t^{\beta - 2} b
\end{displaymath}
for $t \in [\beta,p]$ or, denoting $G =  g^{\beta/2}$,
\begin{displaymath}
 \frac{d}{dt} G \le \frac{\beta}{2}(G^{(\beta -2)/\beta} a + t^{\beta-2}b).
\end{displaymath}
For $\varepsilon > 0$ consider now the function $H_\varepsilon(t) = (g(\beta) + a (t-\beta) + b^{2/\beta} t^{2 - 2/\beta}+\varepsilon)^{\beta/2}$. We have
\begin{displaymath}
H_\varepsilon(\beta) > G(\beta)
\end{displaymath}
and
\begin{align*}
\frac{d}{dt} H_\varepsilon(t) = \frac{\beta}{2} H_\varepsilon(t)^{(\beta-2)/\beta}
(a + (2-2/\beta)t^{1-2/\beta}b^{2/\beta})
\ge \frac{\beta}{2}(H_\varepsilon(t)^{(\beta-2)/\beta}a + t^{\beta -2}b),
\end{align*}
where we used the assumption $\beta \ge 2$. Using the last three inequalities together with the fact that for $t \ge 0$ the function $x \mapsto x^{(\beta-2)/2}a + t^{\beta-2}b$ is increasing on $[0,\infty)$ we obtain that $G(t) \le H_\varepsilon(t)$ for all $t \in [\beta,p]$, which by taking $\varepsilon \to 0^+$ implies that for $p \ge \beta$,
\begin{displaymath}
g(p) = G(p)^{2/\beta} \le H_0(p)^{2/\beta} \le g(\beta) + \frac{D_{LS_\beta}}{2}(p-\beta)\big\||\nabla f(X)|_2\big\|_p^2 + \frac{D_{LS_\beta}^{2/\beta}}{2} p^{2-2/\beta}\big\||\nabla f(X)|_\beta\big\|_p^2,
\end{displaymath}
i.e.,
\begin{equation}\label{eq:intermediate_ineq}
\|f(X)\|_p^2 \le \|f(X)\|_\beta^2 + \frac{D_{LS_\beta}}{2}(p-\beta)\big\||\nabla f(X)|_2\big\|_p^2 + \frac{D_{LS_\beta}^{2/\beta}}{2} p^{2-2/\beta}\big\||\nabla f(X)|_\beta\big\|_p^2.
\end{equation}

The above inequality has been proved so far for strictly positive, locally Lipschitz functions (the boundedness assumption can be easily removed by truncation and passage to the limit). For the case of a general locally Lipschitz function $f$, take any $\varepsilon>0$ and consider $\tilde{f} = |f| + \varepsilon$. Since $\tilde{f}$ is strictly positive and locally Lipschitz, the above inequality holds also for $\tilde{f}$. Taking $\varepsilon \to 0^+$, we can now extend \eqref{eq:intermediate_ineq} to arbitrary locally Lipschitz $f$.

Finally, assume $f \colon \R^{mk} \to \R$ is locally Lipschitz and $f(X)$ is integrable.
Applying~\eqref{eq:intermediate_ineq} to $f - \E f(X)$ instead of $f$ and taking the square root, we obtain
\begin{displaymath}
\|f(X) - \E f(X)\|_p \le \|f(X) - \E f(X)\|_\beta + \sqrt{D_{LS_\beta}(p-\beta)}\big\||\nabla f(X)|_2\big\|_p +
D_{LS_\beta}^{1/\beta} p^{1/\alpha}\big\||\nabla f(X)|_\beta\big\|_p
\end{displaymath}
for $p \ge \beta$. For $p \in [2, \beta]$, since \eqref{eq:modifiedLS} implies the Poincar\'e inequality with constant $D_{LS_\beta}/2$ (see Proposition 2.3. in \cite{Gentil-Guillin-Miclo-2005}), we get
\[
\|f(X) - \E f(X)\|_p \le C D_{LS_\beta}^{1/2} p\big\| |\nabla f(X)|_2\big\|_p
\]
(see the remark following \eqref{eq:Poincare}). These two estimates yield~\eqref{ineq:sobolev-ineq-from-beta-modified} with $C_\beta = C \sqrt{\beta}$.
\end{proof}

\subsection{Applications of Theorem \ref{thm:main_intro}}

Let us now present certain applications of estimates established in the previous section. For simplicity we will restrict to the basic setting presented in Theorem \ref{thm:main_intro}.

\subsubsection{Polynomials} A typical application of Theorem \ref{thm:main_intro} would be to obtain tail inequalities for multivariate polynomials in the random vector $X$. The constants involved in such estimates do not depend on the dimension, but only on the degree of the polynomial. As already mentioned in the introduction, our results in this setting can be considered a transference of inequalities by Lata{\l}a from the tetrahedral Gaussian case to the case of non-necessarily product random vectors and general polynomials.

\subsubsection{Additive functionals and related statistics}
We will now consider three related classes of additive statistics of a random vector, often arising in various problems.

\paragraph{Additive functionals}
Let $X$ be a random vector in $\R^n$ satisfying \eqref{eq:sobolev_def}. For a function $f\colon \R \to \R$ define the random variable
\begin{align}\label{eq:additive_def}
Z_f = f(X_1)+\ldots+f(X_n).
\end{align}

It is classical and follows from \eqref{eq:sobolev_def} by a simple application of the Chebyshev inequality that if $f$ is smooth with $\|f'\|_\infty \le \alpha$, then for all $t > 0$,
\begin{align}\label{eq:additive_D1}
\p\big(|Z_f - \E Z_f| \ge t\big) \le e^2\exp\Big(-\frac{t^2}{e^2 nL^2\alpha^2}\Big).
\end{align}

Using Theorem \ref{thm:main_intro} we can easily obtain inequalities which hold if $f$ is a polynomial-like function, i.e., if $\|f^{(D)}\|_\infty < \infty$ for some $D$. Note that the derivatives of the function $F(x_1,\ldots,x_n) = f(x_1)+\ldots+f(x_n)$ have a very simple diagonal form. In consequence, calculating their $\|\cdot\|_\mathcal{J}$ norms is simple. More precisely, we have
\begin{displaymath}
\D^d F(x) = {\rm diag}_d \Big(f\ub{d}(x_1),\ldots,f\ub{d}(x_n)\Big),
\end{displaymath}
where ${\rm diag}_d(x_1,\ldots,x_n)$ stands for the $d$-indexed matrix $(a_{\ii})_{\ii\in[n]^d}$ such that $a_\ii  = x_i$ if $i_1 = \ldots = i_d = i$ and $0$ otherwise. It is easy to see that if $\mathcal{J} = \{[d]\}$, then $\|{\rm diag}_d(x_1,\ldots,x_n)\|_\mathcal{J} = \sqrt{x_1^2+\ldots+x_n^2}$ and if $\# \mathcal{J} \ge 2$, then
$\|{\rm diag}_d(x_1,\ldots,x_n)\|_\mathcal{J} = \max_{i\le n}|x_i|$. Therefore we obtain the following corollary to Theorem \ref{thm:main_intro}. We will apply it in the next section to linear eigenvalue statistics of random matrices.
\begin{cor}\label{cor:additive_1}
Let $X$ be a random vector in $\R^n$ satisfying \eqref{eq:sobolev_def}, $f \colon \R \to \R$ a $\mathcal{C}^D$ function, such that $\|f^{(D)}\|_\infty < \infty$ and $Z_f$ is defined by \eqref{eq:additive_def}. Then for all $t > 0$,
\begin{align*}
\p(|Z_f - \E Z_f| \ge t) &\le 2\exp\Big(-\frac{1}{C_D}\min\Big(\frac{t^2}{L^{2D}n\|f^{(D)}\|_\infty^2},\frac{t^{2/D}}{L^2\|f^{(D)}\|_\infty^{2/D}}\Big)\Big)\\
&+2\exp\Big(-\frac{1}{C_D}\min_{1\le d\le D-1}\Big(\frac{t^2}{L^{2d} \sum_{i=1}^n (\E f^{(d)}(X_i))^2 }\Big)\Big)\\
&+ 2\exp\Big(-\frac{1}{C_D}\min_{2\le d \le D-1}\Big(\frac{t^{2/d}}{L^2 \max_{i\le n} |\E f^{(d)}(X_i)|^{2/d}}\Big)\Big).
\end{align*}
\end{cor}
Clearly  the case $D=1$ of the above corollary recovers up to constants \eqref{eq:additive_D1}. Moreover using the (yet unproven) Theorem \ref{thm:Gaussian_intro} one can see that for $f(x) = x^D$ and $X$ being a standard Gaussian vector in $\R^n$, the estimate of the corollary is optimal up to absolute constants (in this case, since $Z_f$ is a sum of independent random variables, one can also use estimates from \cite{HOMS}).
\paragraph{Additive functionals of partial sums}
Let us now consider a slightly more involved additive functional of the form
\begin{align}\label{eq:additive_def_2}
S_f = \sum_{i=1}^n f\Big(\sum_{j=1}^i X_j\Big).
\end{align}
Such random variables arise e.g., in the study of additive functionals of random walks (see e.g. \cite{SkorohodSlobodenjuk,BorodinIbragimov}).
For simplicity we will only discuss what can be obtained directly for Lipschitz functions $f$ and what Theorem \ref{thm:main_intro} gives for $f$ with bounded second derivative. Let  thus $F(x) = \sum_{i=1}^n f(\sum_{j=1}^i x_j)$. We have $\frac{\partial}{\partial x_i} F(x) = \sum_{l\ge i} f'(\sum_{j\le l} x_j)$.
Therefore
\begin{displaymath}
\big\||\nabla F|\big\|_\infty^2 = \|f'\|_\infty^2 \sum_{i=1}^n (n-i+1)^2 = \frac{1}{6} n(n+1)(2n+1) \|f'\|_\infty^2,
\end{displaymath}
which, when combined with \eqref{eq:sobolev_def} and Chebyshev's inequality yields
\begin{displaymath}
\p(|S_f - \E S_f| \ge t) \le 2\exp\Big(-\frac{t^2}{C L^2 n^3 \|f'\|_\infty^2}\Big).
\end{displaymath}

Now, let us assume that $f \in \mathcal{C}^2$ and $f''$ is bounded. We have
\begin{displaymath}
|\E \nabla F(X)|^2 = \sum_{i=1}^n\bigg(\sum_{l=i}^n \E f'\Big(\sum_{j=1}^l X_j\Big)\bigg)^2
\end{displaymath}
Moreover
\begin{displaymath}
\frac{\partial^2}{\partial x_i \partial x_j} F(x_1,\ldots,x_n) = \sum_{l \ge i \vee j}f''\Big(\sum_{k=1}^l x_k\Big)
\end{displaymath}
and thus
\begin{align*}
\|\D^2 F(x)\|_{\{1,2\}}^2= \sum_{i,j=1}^n \Big(\sum_{l = i \vee j}^n f''\Big(\sum_{k=1}^l x_k\Big)\Big)^2 \le 2\|f''\|_\infty^2
\sum_{i=1}^n\sum_{j=i}^n (n-j+1)^2 \le Cn^4\|f''\|_\infty^2.
\end{align*}
Since $\D^2 F$ is a symmetric bilinear form, we have
\begin{align*}
\|\D^2 F(x)\|_{\{1\}\{2\}} &\le \sup_{|\alpha| \le 1} \sum_{i,j=1}^n \sum_{l = i \vee j}^n\Big|f''\Big(\sum_{k=1}^l x_k\Big)\Big|\alpha_i\alpha_j\\
&\le \sup_{|\alpha| \le 1} \|f''\|_\infty \sum_{l=1}^n \big(\sum_{i \le l} \alpha_i\big)^2 \le
\sup_{|\alpha| \le 1} \|f''\|_\infty \sum_{l=1}^n l\sum_{i\le l} \alpha_i^2 \le Cn^2 \|f''\|_\infty.
\end{align*}
Using the above estimates and Theorem \ref{thm:main_intro} we obtain
\begin{align*}
\p(|S_f - \E S_f| \ge t) \le 2\exp\Big(-\frac{1}{C L^2}\min\Big(\frac{t^2}{\sum_{i=1}^n \big( \sum_{l=i}^n \E f'(\sum_{j=1}^l X_j) \big)^2 },\frac{t}{n^2\|f''\|_\infty}\Big)\Big).
\end{align*}
To effectively bound the sub-Gaussian coefficient in the above inequality one should use some additional information about the structure of the vector $X$. For a given function $f$ it is of order at most $n^5,$ but if, e.g., the function $f$ is even and $X$ is symmetric, it clearly vanishes. In this case we get
\begin{displaymath}
\p(|S_f - \E S_f| \ge t) \le 2\exp\Big(-\frac{1}{CL^2}\frac{t}{n^2\|f''\|_\infty}\Big).
\end{displaymath}
One can check that if for instance $X$ is a standard Gaussian vector in $\R^n$ and $f(x) = x^2$ then this estimate is tight up to the value of the constant $C$.

\paragraph{$U$-statistics} Our last application in this section will concern $U$-statistics (for simplicity of order 2) of the random vector $X$, i.e., random variables of the form
\begin{displaymath}
U = \sum_{i,j \le n, i\neq j} h_{ij}(X_i,X_j),
\end{displaymath}
where $h_{ij}\colon \R^2 \to \R$ are smooth functions. Without loss of generality let us assume that $h_{ij}(x,y) = h_{ji}(y,x)$.

A simple application of Chebyshev's inequality and \eqref{eq:sobolev_def} gives that if $\D h_{i,j}$ are uniformly bounded on $\R^2$ then for all $t > 0$,
\begin{align*}
\p(|U - \E U| \ge t) &\le 2 \exp\Big(-\frac{1}{C L^2}\frac{t^2}{\sum_{i=1}^n (\sum_{j\neq i} \frac{\partial}{\partial x}h_{ij}(x_i,x_j))^2}\Big) \\
&\le 2\exp\Big(-\frac{1}{C L^2}\frac{t^2}{n^3 \max_{i\neq j}\|\frac{\partial}{\partial x} h_{ij}\|_\infty^2}\Big).
\end{align*}

For $h_{ij}$ of class $\mathcal{C}^2$ with bounded derivatives of second order, a direct application of Theorem \ref{thm:main_intro} gives
\begin{align*}
\p(|U - \E U| \ge t) \le 2\exp\Big(-\frac{1}{C}\min\Big(\frac{t^2}{L^4 \alpha^2},\frac{t^2}{L^2 \beta^2},\frac{t}{L^2 \gamma}\Big)\Big),
\end{align*}
where
\begin{align*}
\alpha^2 & = \sup_{x\in \R^n} \Big\{\sum_{i,j\le n,i\neq j} \Big(\frac{\partial^2}{\partial x\partial y} h_{ij}(x_i,x_j)\Big)^2 + \sum_{i=1}^n\Big(\sum_{j\neq i} \frac{\partial^2}{\partial x^2} h_{ij}(x_i,x_j)\Big)^2\Big\}\\
&\le n^2\max_{i\neq j}\Big\|\frac{\partial^2}{\partial x\partial y} h_{ij}\Big\|_\infty + n^3\max_{i\neq j}\Big\|\frac{\partial^2}{\partial x^2} h_{ij}\Big\|_\infty,\\
\beta^2 & = \sum_{i=1}^n \Big(\sum_{j\neq i}\E \frac{\partial}{\partial x} h_{ij}(X_i,X_j)\Big)^2\le n^3 \max_{i\neq j} |\E \frac{\partial}{\partial x} h_{ij}(X_i,X_j)|^2,\\
\gamma & = \sup_{x \in \R^n} \sup_{|\alpha|,|\beta| \le 1} \Big\{\sum_{i,j\le n,i\neq j} \frac{\partial^2}{\partial x\partial y} h_{ij}(x_i,x_j)\alpha_i\beta_j + \sum_{i=1}^n\alpha_i\beta_i \sum_{j\neq i} \frac{\partial^2}{\partial x^2} h_{ij}(x_i,x_j)\Big\}\\
&\le n \Big(\max_{i\neq j}\Big\|\frac{\partial^2}{\partial x\partial y} h_{ij}\Big\|_\infty + \max_{i\neq j} \Big\|\frac{\partial^2}{\partial x^2} h_{ij}\Big\|_\infty\Big).
\end{align*}

In particular, if $h_{ij} = h$, a function with bounded derivatives of second order, we get $\alpha^2 = \mathcal{O}(n^3)$, $\beta^2 = \mathcal{O}(n^3)$, $\gamma = \mathcal{O}(n)$, which shows that the oscillations of $U$ are of order at most $\mathcal{O}(n^{3/2})$. In the case of $U$-statistics of independent random variables, generated by bounded $h$, this is a well known fact, corresponding to the CLT and classical Hoeffding inequalities for $U$-statistics. We remark that in the so called non-degenerate case, i.e. when $\Var (\E_X h(X,Y)) > 0$, $n^{3/2}$ is then indeed the right normalization in the CLT for $U$-statistics (see e.g. \cite{dlPg}).

\subsubsection{Linear statistics of eigenvalues of random matrices \label{subsubsec:matrices}}

We will now use Corollary \ref{cor:additive_1} to obtain tail inequalities for linear eigenvalue statistics of random Wigner matrices. We remark that one could also apply to the random matrix case the other inequalities considered in the previous section, obtaining in particular estimates on $U$-statistics of eigenvalues (which have been recently investigated by Lytova and Pastur \cite{LP}). We will focus on linear eigenvalues statistics (additive functionals in the language of the previous section) and obtain inequalities involving as a sub-Gaussian term a Sobolev norm of the function $f$ with respect to the semicircle law (the limiting spectral distribution for Wigner ensembles). We refer the reader to the monographs \cite{AGZ,BS,Mehta,PSbook} for basic facts concerning random matrices.

Consider thus a real symmetric $n\times n$ random matrix $A$ ($n \ge 2$) and let $\lambda_1\le \ldots\le \lambda_n$ be its eigenvalues. We will be interested in concentration inequalities for functionals of the form
\begin{displaymath}
Z = \sum_{i=1}^n f(\lambda_i/\sqrt{n}).
\end{displaymath}
In \cite{GZConc} Guionnet and Zeitouni obtained concentration inequalities for $Z$ with Lipschitz $f$ assuming that the entries of $A$ are independent and satisfy the log-Sobolev inequality with some constant $L$. More specifically, they prove that for all $t > 0$,
\begin{align*}
\p(|Z - \E Z| \ge t) \le 2\exp\Big(-\frac{t^2}{8L\|f'\|_\infty^2}\Big).
\end{align*}
(In fact they treat a more general case of banded matrices, but for simplicity we will focus on the basic case.)

As a corollary to Theorem \ref{thm:main_intro} we present below an inequality which compliments the above result. Our aim is to replace the strong parameter $\|f'\|_\infty$ controlling the
sub-Gaussian tail by a weaker Sobolev norm with respect to the semicircular law
\begin{displaymath}
d\rho(x) = \frac{1}{2\pi} \sqrt{4 - x^2} \Ind{(-2,2)}(x)\,dx.
\end{displaymath}
(recall that this is the limiting spectral distribution for Wigner matrices). Imposing additional smoothness assumptions on the function $f$ it can be done in a window $|t| \le c_f n$, where $c_f$ depends on $f$.

\begin{prop}\label{prop:linear-statistics}
Assume the entries of the matrix $A$ are independent (modulo symmetry conditions), mean zero and variance one random variables, satisfying the logarithmic Sobolev inequality~\eqref{eq_log_Sobolev} with constant $L^2$. If $f$ is $\mathcal{C}^2$ with bounded second derivative, then for all $t > 0$,
\begin{align}\label{ineq:linear-statistics-prop}
  \p(|Z - \E Z| \ge t) \le 2 \exp \left( - \frac{1}{C_L} \left(
    \frac{t^2}{\int_{-2}^2 f'^2 \, d\rho +  n^{-2/3} \norm{f''}_\infty^2} \land \frac{nt}{\norm{f''}_\infty} \right) \right). 
\end{align}
\end{prop}

\paragraph{Remark}
The case $f(x) = x^2$ shows that under the assumptions of Proposition~\ref{prop:linear-statistics} one cannot expect a tail behaviour better than exponential for large $t$.
Indeed, since $Z = \frac{1}{n}(\lambda_1^2 + \ldots + \lambda_n^2) = \frac{1}{n}\sum_{i,j \le n} A_{ij}^2$,
even if $A$ is a matrix with standard Gaussian entries, then for all $t > 0$, $\p(|Z-\E Z| \ge t) > \frac{1}{C} \exp( -C (t^2 \land nt))$.

\paragraph{Remark}
A similar inequality to~\eqref{ineq:linear-statistics-prop} holds in the case of Hermitian matrices with independent entries as well. In the proof given below one should invoke an appropriate result concerning the speed of convergence of the spectral distribution of Wigner matrices to the semicircular law.

\begin{proof}
Let us identify the random matrix $A$ with a random vector $\tilde{A} = (A_{ij})_{1\le i\le j\le n}$ having values in $\R^{n(n+1)/2}$ endowed with the standard Euclidean norm $|\tilde{A}| = \left( \sum_{1 \le i \le j \le n} A_{ij}^2\right)^{1/2}$. Note that $\|A\|_{\textup{HS}} \le \sqrt{2} |\tilde{A}|$.
By independence of coordinates of $\tilde{A}$ and the tensorization property of the logarithmic Sobolev inequality (see, e.g., \cite[Corollary 5.7]{LedouxConcBook}), $\tilde{A}$ also satisfies~\eqref{eq_log_Sobolev} with constant $L^2$. Furthermore, by the Hoffman-Wielandt inequality (see, e.g., \cite[Lemma 2.1.19]{AGZ}) which asserts that if $B, C$ are two $n \times n$ real symmetric (or Hermitian) matrices and $\lambda_i(B), \lambda_i(C)$ resp. their eigenvalues arranged in nondecreasing order, then
\begin{displaymath}
\sum_{i=1}^n |\lambda_i(B) - \lambda_i(C)|^2 \le \|B - C\|_{\textup{HS}}^2,
\end{displaymath}
the map $\tilde{A} \mapsto (\lambda_1/\sqrt{n}, \ldots, \lambda_n/\sqrt{n}) \in \R^n$ is $\sqrt{2/n}$-Lipschitz. Therefore, the random vector $(\lambda_1/\sqrt{n}, \ldots, \lambda_n/\sqrt{n})$ satisfies~\eqref{eq_log_Sobolev} with constant $2L^2/n$. In consequence, by the results from~\cite{Aida-Stroock-1994} (see also Theorem~\ref{thm:AidaStroock}), $(\lambda_1/\sqrt{n}, \ldots, \lambda_n/\sqrt{n})$ also satisfies~\eqref{eq:sobolev_def} with constant $L/\sqrt{n}$. Applying Corollary~\ref{cor:additive_1} with $D=2$ we obtain
\begin{equation}\label{ineq:linear-statistics-d-2}
  \p(|Z - \E Z| \ge t) \le 2\exp \left(-\frac{1}{C L^2} \left( \frac{t^2}{n^{-1}\sum_{i=1}^n (\E f'(\lambda_i/\sqrt{n}))^2 + L^2 n^{-1} \norm{f''}^2_\infty} \land \frac{n t}{\norm{f''}_\infty}\right) \right).
\end{equation}

In what follows we shall estimate from above the term $n^{-1}\sum_{i=1}^n (\E f'(\lambda_i/\sqrt{n}))^2$ from~\eqref{ineq:linear-statistics-d-2}. First, by Jensen's inequality
\begin{equation}\label{ineq:rm-sobolev-norm-1}
\frac{1}{n}\sum_{i=1}^n (\E f'(\lambda_i/\sqrt{n}))^2 \le \E \left(\frac{1}{n}\sum_{i=1}^n f'(\lambda_i/\sqrt{n})^2 \right) = \int_\R (f')^2 d\mu,
\end{equation}
where $\mu$ is the expected spectral measure of the matrix $n^{-1/2}A$. According to Wigner's theorem, for a fixed $f$, $\mu$ converges to the semicircular law as $n \to \infty$ and thus $\int_\R (f')^2 \, d\mu \to \int_{-2}^2 (f')^2 \, d\rho$. A non-asymptotic bound on the term $\int_\R f'^2 \, d\mu$ can be obtained using the result of Bobkov, G{\"o}tze and Tikhomirov~\cite{Bobkov-Goetze-Tikhomirov-2010} on the speed of convergence of the expected spectral distribution of real Wigner matrices to the semicircular law. Since each entry of $A$ satisfies the logarithmic Sobolev inequality with constant $L^2$, it also satisfies the Poincar{\'e} inequality with the same constant (see e.g.~\cite[Chapter 5]{LedouxConcBook}). Therefore Theorem 1.1 from~\cite{Bobkov-Goetze-Tikhomirov-2010} gives
\begin{equation}\label{ineq:bgt-kolmogorov-distance}
  \sup_{x \in \R} |F_\mu(x) - F_\rho(x)| \le C_L n^{-2/3},
\end{equation}
where $F_\mu$ and $F_\rho$ are the distribution functions of $\mu$ and $\rho$, respectively.

The decay of $1-F_\mu(x)$ and $F_\mu(x)$ as $x \to \infty$ and $x \to -\infty$ (resp.) can be obtained using the sub-Gaussian concentration of $\lambda_n/\sqrt{n}$ and $\lambda_1/\sqrt{n}$, which is, e.g., a consequence of~\eqref{eq:sobolev_def} for the vector of eigenvalues of $n^{-1/2} A$. For example, for any $t \ge 0$,
\begin{align}\label{ineq:decay-of-F-1}
  \p\left(\frac{\lambda_n}{\sqrt{n}} \ge \E \frac{\lambda_n}{\sqrt{n}} + t \right) &\le 2 \exp\left(- \frac{1}{C} \frac{n t^2}{L^2} \right).
\end{align}
Using the classical technique of $\delta$-nets for estimating the operator norm of a matrix (see e.g.~\cite{PisierVolume}) and the fact that the entries of $A$ are sub-Gaussian (as they satisfy the logarithmic Sobolev inequality) one gets
$\E \lambda_n \le \E \|A\|_{\text{op}} \le C L \sqrt{n}$, which together with~\eqref{ineq:decay-of-F-1} yields
\begin{equation}\label{ineq:decay-of-F}
  1 - F_\mu(CL + t) \le \p\left(\frac{\lambda_n}{\sqrt{n}} \ge CL + t \right) \le 2 \exp\left(- \frac{1}{C} \frac{n t^2}{L^2} \right)
\end{equation}
for all $t \ge 0$. Clearly, the same inequality holds for $F(-CL - t)$.
Integrating by parts,
\begin{equation}\label{eq:bobkov-goetze-tikhomirov-int-by-parts}
  \int_\R f'^2 \, d\mu = \int_\R f'^2 \, d\rho + \int_\R \left( f'(x)^2 \right)' (F_\rho(x) - F_\mu(x)) \, dx.
\end{equation}
Combining the uniform estimate~\eqref{ineq:bgt-kolmogorov-distance} with~\eqref{ineq:decay-of-F} and using an elementary inequality $2 x y \le x^2 + y^2$, we estimate the last integral in~\eqref{eq:bobkov-goetze-tikhomirov-int-by-parts} as follows:
\begin{multline}\label{ineq:bgt-error-term}
  \left| \int_\R \left(f'(x)^2 \right)' (F_\mu(x) - F_\rho(x)) \, dx \right| \\[1ex]
\le \int_\R \left| 2 f'(x) f''(x) \right| \left( \norm{F_\mu - F_\rho}_\infty \land 2\exp\left(-\frac{n}{C} \frac{ \text{dist}(x, [-CL, CL])^2}{L^2} \right) \right) \, dx \\[1ex]
\le \int_\R f'(x)^2 \, d\nu(x) + \nu(\R) \norm{f''}_\infty^2,
\end{multline}
where
\[
  d\nu(x) = C_L n^{-2/3} \land 2\exp\left(-\frac{\text{dist}(x, [-CL, CL])^2}{2 \sigma^2}\right) \, dx, \qquad \text{and} \qquad \sigma^2 = \frac{C L^2}{2n}.
\]

We proceed to estimate the two last terms from~\eqref{ineq:bgt-error-term}.
Take $r > 0$ such that
\begin{align}\label{eq:on-r-1}
  2 e^{-r^2/(2\sigma^2)} = C_L n^{-2/3}
\end{align}
or put $r=0$ if no such $r$ exists. Note that if we assume $C_L \ge 1$, as we obviously can, then
\begin{align}\label{eq:on-r-2}
  r \le C L n^{-1/2} \sqrt{\log n}.
\end{align}
We shall need the following estimates, which are easy consequences of the standard estimate for a Gaussian tail:
\begin{align}\label{ineq:gaussian-tail-1}
  \int_r^\infty e^{-y^2/(2\sigma^2)} \, dy \le C \sigma e^{-r^2/(2\sigma^2)} \le C_L \sigma n^{-2/3} \le C_L n^{-7/6},
\end{align}
and
\begin{equation}\label{ineq:gaussian-tail-2}
  \begin{split}
  \int_r^\infty y^2 e^{-y^2/(2\sigma^2)} \, dy &\le \left( \int_0^\infty y^4 e^{-y^2/(2\sigma^2)} \, dy \right)^{1/2}
  \left( \int_r^\infty e^{-y^2/(2\sigma^2)} \, dy \right)^{1/2} \\[1ex]
  &\le C_L \sigma^{5/2} (\sigma n^{-2/3})^{1/2} \le C_L n^{-11/6}.
  \end{split}
\end{equation}
Now, \eqref{eq:on-r-1}, \eqref{eq:on-r-2} and \eqref{ineq:gaussian-tail-1} yield
\begin{align}\label{ineq:nu-R}
  \nu(\R) \le (CL + r) C_L n^{-2/3} + 4\int_r^\infty e^{-y^2/(2\sigma^2)} \, dy \le C_L n^{-2/3}.
\end{align}
We shall also need the estimate for $\int_\R x^2 \, d\nu(x)$ which follows from~\eqref{eq:on-r-1}, \eqref{eq:on-r-2} and~\eqref{ineq:gaussian-tail-2}:
\begin{align}\label{ineq:nu-x2}
  \int_\R x^2 \, d\nu(x) = \frac23 (CL+r)^3 C_L n^{-2/3} + 4 \int_r^\infty (CL+y)^2 e^{-y^2/(2\sigma^2)} \, dy \le C_L n^{-2/3}.
\end{align}

In order to estimate $\int_\R f'^2 \, d\nu$, take any $x_0 \in [-2,2]$ such that $|f'(x_0)|^2 \le \int_{-2}^2 f'^2 \, d\rho$, and use $|f'(x)| \le |f'(x_0)| + |x-x_0| \norm{f''}_\infty$ to obtain
\begin{align*}
  \int_\R f'(x)^2 \, d\nu(x) &\le 2 \Big( \int_{-2}^2 f'^2 \, d\rho \Big) \nu(\R) + 2\norm{f''}^2_\infty \int_\R |x-x_0|^2 \, d\nu(x) \\[1ex]
  &\le 2 \Big( \int_{-2}^2 f'^2 \, d\rho \Big) \nu(\R) + 4 \norm{f''}^2_\infty x_0^2 \nu(\R) + 4\norm{f''}^2_\infty \int_\R x^2 \, d\nu(x).
\end{align*}
Plugging \eqref{ineq:nu-R} and \eqref{ineq:nu-x2} into the above yields
\begin{align}\label{ineq:nu-f2}
  \int_\R f'(x)^2 \, d\nu(x) \le C_L n^{-2/3} \left( \int_{-2}^2 f'^2 \, d\rho + \norm{f''}_\infty^2 \right).
\end{align}
In turn, plugging~\eqref{ineq:nu-R} and~\eqref{ineq:nu-f2} into~\eqref{ineq:bgt-error-term} and then combining with~\eqref{eq:bobkov-goetze-tikhomirov-int-by-parts} we finally get
\[
  \int_\R f'^2 \, d\mu \le
    (1 + C_L n^{-2/3}) \int_\R f'^2 \, d\rho + C_L n^{-2/3} \norm{f''}_\infty
\]
which combined with~\eqref{ineq:linear-statistics-d-2} and \eqref{ineq:rm-sobolev-norm-1} completes the proof.
\end{proof}

\paragraph{Remark} With some more work (using truncations or working directly on moments) one can extend the above proposition to the case, when $|f''(x)| \le a(1+|x|^k)$ for some non-negative integer $k$ and $a \in \R$. In this case we obtain
\[
  \p\big(|Z-\E Z| \ge t\big)
  \le 2 \exp\left(-\left(\frac{t^2}{C_L \int_{-2}^2 f'^2 \,d\rho + C_{L,k} n^{-2/3} a^2} \land \frac{n}{C_{L,k}} \left(\frac{t}{a}\right)^{\frac{2}{k+2}} \right) \right).
\]
We also remark that to obtain the inequality \eqref{ineq:linear-statistics-d-2} one does not have to use independence of the entries of $A$, it is enough to assume that the vector $\tilde{A}$ satisfies the inequality \eqref{eq:sobolev_def}.

\section{Two-sided estimates of moments for Gaussian polynomials}\label{sec:Gaussian}
We will now prove Theorem \ref{thm:Gaussian_intro}, showing that in the case of general polynomials in Gaussian variables, the estimates of Theorem \ref{thm:main_intro} are optimal (up to constants depending only on the degree of the polynomial). In the special case of tetrahedral polynomials this follows from Lata{\l}a's Theorem \ref{thm:Latala_intro} and the following result by Kwapie\'n.

\begin{theorem}[Kwapie\'n, \cite{KwaDec}]\label{thm_Kwapien}
If $X = (X_1,\ldots,X_n)$ where $X_i$ are independent symmetric random variables, $Q$ is a multivariate tetrahedral polynomial of degree $D$ with coefficients in a Banach space $E$ and $Q_d$ is its homogeneous part of degree $d$, then for any symmetric convex function $\Phi \colon E \to \R_+$ and any $d \in \{0,1, \ldots, D\}$,
\begin{displaymath}
\E\Phi(Q_d(X)) \le \E\Phi(C_d Q(X)). 
\end{displaymath}
\end{theorem}

Indeed, when combined with Theorem \ref{thm:Latala_intro} and the triangle inequality, the above theorem gives the following

\begin{cor} \label{cor_tetrahedral} Let
\begin{displaymath}
Z = \sum_{0\le d \le D} \sum_{\ii \in [n]^d} a\ub{d}_\ii  g_{i_1}\cdots g_{i_d},
\end{displaymath}
where $A_d = (a\ub{d}_\ii)_{\ii\in [n]^d}$ is a $d$-indexed symmetric matrix of real numbers such that $a_\ii = 0$ if $i_j = i_l$ for some $k\neq l$ (we adopt the convention that for $d=0$ we have a single number $a\ub{0}_\emptyset$). Then for any $p\ge 2$,
\begin{displaymath}
C_D^{-1}\sum_{0\le d\le D} \sum_{\mathcal{J} \in P_d} p^{\#\mathcal{J}/2} \|A_d\|_\mathcal{J} \le \|Z\|_p \le C_D \sum_{0\le d\le D} \sum_{\mathcal{J} \in P_d} p^{\#\mathcal{J}/2} \|A_d\|_\mathcal{J}.
\end{displaymath}
\end{cor}

The strategy of proof of Theorem \ref{thm:Gaussian_intro} is very simple and relies on infinite divisibility of Gaussian random vectors, which will help us approximate the law of a general polynomial in Gaussian variables by the law of a tetrahedral polynomial, for which we will use Corollary \ref{cor_tetrahedral}.

\medskip
It will be convenient to have the polynomial $f$ represented as a combination of multivariate Hermite polynomials:
\begin{equation}\label{eq:f-as-Hermite}
  f(x_1, \ldots, x_n) = \sum_{d=0}^D \sum_{\dd \in \Delta_d^n} a_\dd h_{d_1}(x_1) \cdots h_{d_n}(x_n),
\end{equation}
where
\[
  \Delta_d^n = \{ \dd = (d_1, \ldots, d_n) \colon \forall_{k \in [n]}\ d_k \ge 0 \text{ and } d_1 + \cdots + d_n = d \}
\]
and $h_d(x) = (-1)^d e^{x^2/2} \frac{d^n}{dx^n} e^{-x^2/2}$ is the $d$-th Hermite polynomial.

Let $(W_t)_{t \in [0,1]}$ be a standard Brownian motion. Consider standard Gaussian random variables $g = W_1$ and, for any positive integer $N$,
\[
  g_{j,N} = \sqrt{N} (W_{\frac{j}{N}} - W_{\frac{j-1}{N}}), \quad j = 1, \ldots, N.
\]
For any $d \ge 0$, we have the following representation of $h_d(g) = h_d(W_1)$ as a multiple stochastic integral
(see~\cite[Example 7.12 and Theorem 3.21]{JansonGHS}),
\[
  h_d(g) = d! \int_0^1 \! \int_0^{t_d} \! \cdots \! \int_0^{t_2} \, dW_{t_1} \cdots dW_{t_{d-1}} dW_{t_d}.
\]
Approximating the multiple stochastic integral leads to
\begin{equation}\label{eq:Hermite-as-tetrahedral-polynomial}
  \begin{split}
  h_d(g) &= d! \lim_{N \to \infty} N^{-d/2} \sum_{1 \le j_1 < \cdots < j_d \le N} g_{j_1, N} \cdots g_{j_d, N} \\[1ex]
         &= \lim_{N \to \infty} N^{-d/2} \sum_{\jj \in [N]\uu{d}} g_{j_1, N} \cdots g_{j_d, N},
\end{split}
\end{equation}
where the limit is in $L^2(\Omega)$ (see \cite[Theorem 7.3. and formula (7.9)]{JansonGHS}) and actually the convergence holds in any $L^p$ (see~\cite[Theorem 3.50]{JansonGHS}). We remark that instead of multiple stochastic integrals with respect to the Wiener process we could use the CLT for canonical $U$-statistics (see \cite[Chapter 4.2]{dlPg}), however the stochastic integral framework seems more convenient as it allows to put all the auxiliary variables on the same probability space.

Now, consider $n$ independent copies $(W_t\ub{i})_{t \in [0,1]}$ of the Brownian motion ($i=1, \ldots, n$) together with the corresponding Gaussian random variables: $g\ub{i} = W_1\ub{i}$ and, for $N \ge 1$,
\[
  g_{j, N}\ub{i} = \sqrt{N} (W_{\frac{j}{N}}\ub{i} - W_{\frac{j-1}{N}}\ub{i}), \quad j = 1, \ldots, N.
\]
In the lemma below we state the representation of a multivariate Hermite polynomial in the variables $g\ub{1}, \ldots, g\ub{n}$ as a limit of tetrahedral polynomials in the variables $g_{j, N}\ub{i}$. To this end introduce some more notation. Let
\[
  G\ub{n, N} = (g_{1,N}\ub{1}, \ldots, g_{N,N}\ub{1}, \ g_{1,N}\ub{2}, \ldots, g_{N,N}\ub{2}, \ \ldots,\  g_{1,N}\ub{n}, \ldots, g_{N,N}\ub{n}) = (g_{j,N}\ub{i})_{(i,j)\in [n]\times [N]}
\]
be a Gaussian vector with $n \times N$ coordinates. We identify here the set $[nN]$ with $[n]\times [N]$ via the bijection $(i,j) \leftrightarrow (i-1)N+j$. We will also identify the sets $([n]\times [N])^d$ and $[n]^d\times [N]^d$ in a natural way. For $d \ge 0$ and $\dd \in \Delta_d^n$, let
\[
  I_{\dd} = \big\{ \ii \in [n]^d \colon \forall_{l \in [n]} \, \# \ii^{-1}(\{l\}) = d_l \big\},
\]
and define a $d$-indexed matrix $B_{\dd}\ub{N}$ of $n^d$ blocks each of size $N^d$ as follows: for $\ii \in [n]^d$ and $\jj \in [N]^d$,
\[
  \big(B_{\dd}\ub{N}\big)_{(\ii, \jj)} = \begin{cases}
    \frac{d_1! \cdots d_n!}{d!} N^{-d/2} & \text{if $\ii \in I_{\dd}$ and $(\ii, \jj) := \big((i_1, j_1), \ldots, (i_d, j_d)\big) \in ([n] \times [N])\uu{d},$}
    \\[1ex]
    0 & \text{otherwise.}
  \end{cases}
\]
\begin{lemma}\label{lemma:multivariate-Hermite-as-tetrahedral-polynomial}
With the above notation, for any $p > 0$,
\[
  \big\langle B_{\dd}\ub{N}, (G\ub{n,N})^{\otimes d} \big\rangle \stackrel[N \to \infty]{}{\longrightarrow} h_{d_1}(g\ub{1}) \cdots h_{d_n}(g\ub{n}) \quad \text{in $L^p(\Omega)$}.
\]
\end{lemma}
\begin{proof}
Using~\eqref{eq:Hermite-as-tetrahedral-polynomial} for each $h_{d_i}(g\ub{i})$,
\begin{multline*}
  h_{d_1}(g\ub{1}) \cdots h_{d_n}(g\ub{n}) \\
  = \lim_{N \to \infty} N^{-d/2} \sum_{\substack{(j_1^{(1)}, \ldots, j_{d_1}^{(1)}) \in [N]^{\underline{d_1}}
\\ \vdots \\
(j_1^{(n)}, \ldots, j_{d_n}^{(n)}) \in [N]^{\underline{d_n}} }}
 \big( g_{j_1\ub{1},N}\ub{1} \cdots g_{j_{d_1}\ub{1},N}\ub{1} \big) \cdots
 \big( g_{j_1\ub{n},N}\ub{n} \cdots g_{j_{d_n}\ub{n},N}\ub{n} \big).
\end{multline*}
For each $N$, the right-hand side equals
\begin{displaymath}
  \frac{1}{\# I_{\dd}} N^{-d/2} \sum_{\ii \in I_{\dd}} \sum_{\substack{\jj \in [N]^d \text{ s.t.} \\ (\ii, \jj) \in ([n]\times[N])\uu{d}}} g_{j_1, N}\ub{i_1} \cdots g_{j_d, N}\ub{i_d} =
  \big\langle B_{\dd}\ub{N}, (G\ub{n,N})^{\otimes d} \big\rangle,
\end{displaymath}
since $\# I_{\dd} = \frac{d!}{d_1! \cdots d_n!}$.
\end{proof}
Note that $B_\dd\ub{N}$ is symmetric, i.e., for any $\ii \in [n]^d$, $\jj \in [N]^d$ if $\pi \colon [d] \to [d]$ is a permutation and $\ii' \in [n]^d$, $\jj' \in [N]^d$ are such that $\forall_{k \in [d]} \; i'_k = i_{\pi(k)}$ and $j'_k = j_{\pi(k)}$, then
\[
  \big( B_\dd\ub{N} \big)_{(\ii',\jj')} =  \big( B_\dd\ub{N} \big)_{(\ii,\jj)}.
\]
Moreover, $B_\dd\ub{N}$ has zeros on ``generalized diagonals'', i.e., $\big( B_\dd\ub{N} \big)_{(\ii,\jj)} = 0$ if $(i_k, j_k) = (i_l, j_l)$ for some $k \neq l$.

\begin{proof}[Proof of Theorem \ref{thm:Gaussian_intro}]Let us first note that it is enough to prove the moment estimates, the tail bound follows from them by the Paley-Zygmund inequality (see e.g. the proof of Corollary 1 in \cite{L2}). Moreover, the upper bound on moments follows directly from Theorem~\ref{thm:main_intro}. For the lower bound we use Lemma~\ref{lemma:multivariate-Hermite-as-tetrahedral-polynomial} to approximate the $L^p$ norm of $f(G)-\E f(G)$ with that of a tetrahedral polynomial, for which we can use the lower bound from Corollary~\ref{cor_tetrahedral}.

Assuming $f$ is of the form~\eqref{eq:f-as-Hermite}, Lemma~\ref{lemma:multivariate-Hermite-as-tetrahedral-polynomial} together with the triangle inequality implies
\[
  \lim_{N \to \infty} \Big\|\sum_{d=1}^D \Big\langle \sum_{\dd \in \Delta_d^n} a_\dd B_\dd\ub{N}, \big(G\ub{n,N}\big)^{\otimes d} \Big\rangle \Big\|_p = \big\|f(G) - \E f(G)\big\|_p
\]
for any $p > 0$, where $G = (g\ub{1}, \ldots, g\ub{n} )$. It therefore remains to relate
$\big\|\sum_{\dd} a_\dd B_\dd\ub{N}\big\|_{\mathcal{J}}$ with $\norm{\E \D^d f(G)}_{\mathcal{J}}$ for any $d \ge 1$ and $\mathcal{J} \in P_d$. In fact we shall prove that
\begin{equation}\label{eq:description-via-diff}
  \lim_{N \to \infty} \Big\|\sum_{\dd \in \Delta_d^n} a_\dd B_\dd\ub{N}\Big\|_{\mathcal{J}} = \frac{1}{d!} \norm{\E \D^d f(G)}_{\mathcal{J}},
\end{equation}
which will end the proof.

Fix $d \ge 1$ and $\mathcal{J} \in P_d$. For any $\dd \in \Delta_d^n$ define a symmetric $d$-indexed matrix $(b_\dd)_{\ii \in [n]^d}$ as
\[
  (b_\dd)_\ii = \begin{cases}
    \frac{d_1! \cdots d_n!}{d!} & \text{if $\ii \in I_\dd,$} \\
    0 & \text{otherwise.}
  \end{cases}
\]
and a symmetric $d$-indexed matrix $(\tilde{B}_\dd\ub{N})_{(\ii, \jj) \in ([n] \times [N])^d}$ as
\[
  (\tilde{B}_\dd\ub{N})_{(\ii, \jj)} = N^{-d/2} (b_\dd)_\ii \quad \text{for all $\ii \in [n]^d$ and $\jj \in [N]^d.$}
\]
It is a simple observation that
\begin{equation}\label{eq:blown-matrices}
  \Big\| \sum_{\dd \in \Delta_d^n} a_\dd \tilde{B}_\dd\ub{N} \Big\|_\mathcal{J} = \Big\| \sum_{\dd \in \Delta_d^n} a_\dd  (b_\dd)_{\ii \in [n]^d} \Big\|_\mathcal{J}.
\end{equation}
On the other hand, for any $\dd \in \Delta_d^n$, the matrices $\tilde{B}_\dd\ub{N}$ and $B_\dd\ub{N}$ differ at no more than $\# I_\dd \cdot \#([N]^d \setminus [N]\uu{d})$ entries. More precisely, if $\mathcal{J}_0 = \{ [d] \}$ (a trivial partition of $[d]$ into one set), then
\[
  \big\| \tilde{B}_\dd\ub{N} - B_\dd\ub{N} \big\|_\mathcal{J}^2 \le \big\| \tilde{B}_\dd\ub{N} - B_\dd\ub{N} \big\|_{\mathcal{J}_0}^2 \le \frac{d_1! \cdots d_n!}{d!} N^{-d} (N^d - N\uu{d}) \longrightarrow 0 \quad \text{as $N \to \infty$}.
\]
Thus the triangle inequality for the $\|\cdot\|_\mathcal{J}$ norm together with~\eqref{eq:blown-matrices} yields
\begin{equation}\label{eq:B-and-b}
  \lim_{N \to \infty} \Big\|\sum_{\dd \in \Delta_d^n} a_\dd B_\dd\ub{N}\Big\|_{\mathcal{J}} = \Big\| \sum_{\dd \in \Delta_d^n} a_\dd  (b_\dd)_{\ii \in [n]^d} \Big\|_\mathcal{J}.
\end{equation}

Finally, note that
\begin{equation}\label{eq:Df-and-b}
  \E \D^d f(G) = d! \sum_{\dd \in \Delta_d^n} a_\dd  (b_\dd)_{\ii \in [n]^d}.
\end{equation}
Indeed, using the identity on Hermite polynomials, $h_k'(x) = k h_{k-1}(x)$ ($k \ge 1$), we obtain $\E h_k\ub{l}(g) = k! \delta_{k,l}$ for $k,l\ge 0$, where $f\ub{l}$ stands for the $l$-th derivative of $f$, and thus, for any $\dd \in \Delta_d^n$,
\[
  \big(\E \D^d h_{d_1}(g\ub{1}) \cdots h_{d_n}(g\ub{n})\big)_\ii = d! (b_\dd)_\ii \quad \text{for each $\ii \in [n]^d$}.
\]
Now, \eqref{eq:Df-and-b} follows by linearity. Combining it with~\eqref{eq:B-and-b} proves~\eqref{eq:description-via-diff}.
\end{proof}

\paragraph{Remark} Note that the above infinite-divisibility argument can be also used to prove the upper bound on moments in Theorem \ref{thm:Gaussian_intro} (giving a proof independent of the one relying on Theorem \ref{thm:main_intro}).

\section{Polynomials in independent sub-Gaussian random variables}\label{sec:subgaussian}

In this section we prove Theorem \ref{thm:subgaussian_intro}. Before we proceed with the core of the proof we will need to introduce some auxiliary inequalities for the norms $\|\cdot\|_\mathcal{J}$ as well as some additional notation.

\subsection{Properties of $\|\cdot\|_\mathcal{J}$ norms}
The first inequality we will need is pretty standard and given in the following lemma  (it is a direct consequence of the definition of the norms $\|\cdot\|_{\mathcal{J}}$).

\begin{lemma}\label{lem_tensor_product}
For any $d$-indexed matrix $A = (a_\ii)_{\ii \in [n]^d}$ and any vectors $v_1,\ldots,v_d \in \R^n$ we have for all $\mathcal{J} \in P_d$,
\begin{displaymath}
\|A\circ \otimes_{i=1}^d v_i\|_\mathcal{J} \le \|A\|_\mathcal{J}\prod_{i=1}^d \|v_i\|_\infty
\end{displaymath}
\end{lemma}

To formulate subsequent inequalities we need some auxiliary notation concerning $d$-indexed matrices. We will treat matrices as functions from $[n]^d$ into the real line, which in particular allows us to use the notation of indicator functions and for a set $C \subseteq \{1,\ldots,n\}^d$ write $\Ind{C}$ for the matrix $(a_\ii)$ such that $a_\ii = 1$ if $\ii\in C$ and $0$ otherwise.

 Note that for $\#\mathcal{J} > 1$, $\|\cdot\|_\mathcal{J}$ is not unconditional in the standard basis, i.e., in general it is not true that $\|A\circ\Ind{C}\|_\mathcal{J} \le \|A\|_\mathcal{J}$. One situation in which this inequality holds is when $C$ is of the form $C = \{\ii\colon i_{k_1} = j_1,\ldots,i_{k_l} = j_l\}$ for some $1 \le k_1<\ldots<k_l \le d$ and $j_1,\ldots,j_l \in [n]$ (which follows from Lemma \ref{lem_tensor_product}). This corresponds to setting to zero all coefficients which are outside a ``generalized row'' of a matrix and leaving the coefficients in this row intact.

Later we will need another inequality of this type, which will allow us to select a ``generalized diagonal'' of a matrix. The corresponding estimate is given in the following

\begin{lemma}\label{lem_diagonal_selection}
Let $A = (a_\ii)_{\ii\in [n]^d}$ be a $d$-indexed matrix and let $C \subseteq [n]^d$ be of the form $C = \{\ii\colon i_k = i_l \;\textrm{for}  \; k,l \in K\}$, with $K \subseteq [d]$. Then for every $\mathcal{J} \in P_d$, $\|A\circ \Ind{C}\|_\mathcal{J} \le \|A\|_\mathcal{J}$.
\end{lemma}

\begin{proof}
Since $\Ind{C_1\cap C_2} = \Ind{C_1} \circ \Ind{C_2}$, it is enough to consider the case $\#K = 2$, i.e. $C = \{\ii\colon i_k = i_l\}$ for some $1\le k<l\le d$.
Let $\mathcal{J} = \{J_1,\ldots,J_m\}$. We will consider two cases.

\paragraph{1.} The numbers $k$ and $l$ are separated by the partition $\mathcal{J}$. Without loss of generality we can assume that $k\in J_1$, $l\in J_2$. Then
\begin{align}\label{eq_diag_selection}
&\|A\circ \Ind{C}\|_\mathcal{J} \\
&= \sup_{\|x\ub{j}_{\ii_{J_j}}\|_2\le 1\colon  j \ge 3}\Big(\sup_{\|x\ub{1}_{\ii_{J_1}}\|_2,\|x\ub{2}_{\ii_{J_2}}\|_2\le 1} \sum_{|\ii_{J_1}|\le n}\sum_{|\ii_{J_2}|\le n}\ind{i_k=i_l}\Big(\sum_{|\ii_{(J_1\cup J_2)^c}|\le  n} a_\ii x\ub{3}_{\ii_{J_3}}\cdots x\ub{m}_{\ii_{J_m}}\Big)x\ub{1}_{\ii_{J_1}}x\ub{2}_{\ii_{J_2}}\Big).\nonumber
\end{align}

For any $x\ub{3}_{\ii_{J_3}},\ldots,x\ub{m}_{\ii_{J_m}}$, consider the matrix
\begin{displaymath}
B_{\ii_{J_1},\ii_{J_2}} = \Big(\sum_{|\ii_{(J_1\cup J_2)^c}|\le  n} a_\ii x\ub{3}_{\ii_{J_3}}\cdots x\ub{m}_{\ii_{J_m}}\Big)_{\ii_{J_1},\ii_{J_2}}
\end{displaymath}
acting from $\ell_2([n]^{J_1})$ to $\ell_2([n]^{J_2})$.

For fixed $x\ub{3}_{\ii_{J_3}},\ldots,x\ub{m}_{\ii_{J_m}}$ the inner expression on the right hand side of (\ref{eq_diag_selection}) is the operator norm of the block-diagonal matrix obtained from $B_{\ii_{J_1},\ii_{J_2}}$ by setting to zero entries in off-diagonal blocks. Therefore it is not greater than the operator norm of $B_{\ii_{J_1},\ii_{J_2}}$, which allows us to write
\begin{align*}
\|A\circ \Ind{C}\|_\mathcal{J} &\le \sup_{\|x\ub{j}_{\ii_{J_j}}\|_2\le 1\colon  j \ge 3}\Big(\sup_{\|x\ub{1}_{\ii_{J_1}}\|_2,\|x\ub{2}_{\ii_{J_2}}\|_2\le 1} \sum_{|\ii_{J_1}|\le n}\sum_{|\ii_{J_2}|\le n}\Big(\sum_{|\ii_{(J_1\cup J_2)^c}|\le  n} a_\ii x\ub{3}_{\ii_{J_3}}\cdots x\ub{m}_{\ii_{J_m}}\Big)x\ub{1}_{\ii_{J_1}}x\ub{2}_{\ii_{J_2}}\Big)\\
& = \|A\|_\mathcal{J}.
\end{align*}

\paragraph{2.} There exists $j$ such that $k,l \in J_j$. Without loss of generality we can assume that $j = 1$. We have
\begin{align*}
\|A\circ \Ind{C}\|_\mathcal{J}
&= \sup_{\|x\ub{j}_{\ii_{J_j}}\|_2\le 1\colon  j \ge 2}\Big(\sup_{\|x\ub{1}_{\ii_{J_1}}\|_2\le 1} \sum_{|\ii_{J_1}|\le n}\ind{i_k=i_l}\Big(\sum_{|\ii_{J_1^c}|\le  n} a_\ii x\ub{2}_{\ii_{J_2}}\cdots x\ub{m}_{\ii_{J_m}}\Big)x\ub{1}_{\ii_{J_1}}\Big)\\
& = \sup_{\|x\ub{j}_{\ii_{J_j}}\|_2\le 1\colon  j \ge 2}\Big(\sum_{|\ii_{J_1}|\le n}\ind{i_k=i_l}\Big(\sum_{|\ii_{J_1^c}|\le  n} a_\ii x\ub{2}_{\ii_{J_2}}\cdots x\ub{m}_{\ii_{J_m}}\Big)^2\Big)^{1/2}\\
&\le \sup_{\|x\ub{j}_{\ii_{J_j}}\|_2\le 1\colon  j \ge 2}\Big(\sum_{|\ii_{J_1}|\le n}\Big(\sum_{|\ii_{J_1^c}|\le  n} a_\ii x\ub{2}_{\ii_{J_2}}\cdots x\ub{m}_{\ii_{J_m}}\Big)^2\Big)^{1/2} = \|A\|_\mathcal{J}.
\end{align*}

\end{proof}

For a partition $\mathcal{K}  = \{K_1,\ldots,K_m\} \in P_d$ define
\begin{align}\label{eq_level_set_def}
 L(\mathcal{K}) = \{\ii\in[n]^d\colon i_k= i_l \;\textrm{iff}\; \exists_{j\le m}\; k,l\in K_j\}.
\end{align}
Thus $L(\mathcal{K})$ is the set of all indices for which the partition into level sets is equal to $\mathcal{K}$.

\begin{cor} \label{cor_norm_monotonicity} For any $\mathcal{J,K} \in P_d$ and any $d$-indexed matrix $A$,
\begin{displaymath}
\|A\circ \Ind{L(\mathcal{K})}\|_\mathcal{J} \le 2^{\#\mathcal{K}(\#\mathcal{K}-1)/2}\|A\|_\mathcal{J}.
\end{displaymath}
\end{cor}

\begin{proof}
By Lemma \ref{lem_diagonal_selection} and the triangle inequality for any $k< l$, $\|A\circ \ind{i_k\neq i_l}\|_\mathcal{J} = \|A - A\circ\ind{i_k=i_l}\|_\mathcal{J} \le 2\|A\|_\mathcal{J}$. Now it is enough to note that $L(\mathcal{K})$ can be expressed as an intersection of $\#\mathcal{K}$ ``generalized diagonals'' and
$\#\mathcal{K}(\#\mathcal{K}-1)/2$ sets of the form $\{\ii\colon i_k\neq i_l\}$ where $k < l$ and use again Lemma \ref{lem_diagonal_selection} together with the above inequality.
\end{proof}

\subsection{Proof of Theorem \ref{thm:subgaussian_intro}}
Let us first note that the tail bound of Theorem \ref{thm:subgaussian_intro} follows from the moment estimate and Chebyshev inequality in the same way as in Theorems \ref{thm:main_intro} or \ref{thm:main}. We will therefore focus on the moment bound.

The method of proof will rely on the reduction to the Gaussian case via decoupling inequalities, symmetrization and the contraction principle. To carry out this strategy we will need the following representation of $f$.
\begin{align}\label{eq_poly_rep_1}
f(x) = \sum_{0\le d\le D} \sum_{m=0}^d \sum_{{k_1, \ldots, k_m > 0}\atop{k_1+\ldots+k_m=d}} \sum_{\ii \in [n]\uu{m}} c_
{(i_1,k_1),\ldots,(i_m,k_m)}\ub{d} x_{i_1}^{k_1}x_{i_2}^{k_2}\cdots x_{i_m}^{k_m},
\end{align}
where the coefficients  $c_{(i_1,k_1),\ldots,(i_m,k_m)}\ub{d}$ satisfy
\begin{align}\label{eq:poly_symmetry}
c_{(i_1,k_1),\ldots,(i_m,k_m)}\ub{d} = c_{(i_{\pi_1},k_{\pi_1}),\ldots,(i_{\pi_m},k_{\pi_m})}\ub{d}
\end{align}
for all permutations $\pi \colon [m] \to [m]$. At this point we would like to explain the convention regarding indices which we will use throughout this section. It is rather standard, but we prefer to draw the Reader's attention to it, as we will use it extensively in what follows. Namely, we will treat the sequence $\kk = (k_1,\ldots,k_m)$ as a function acting on $[m]$ and taking values in positive integers. In particular if $m=0$, then $[m] = \emptyset$ and there exists exactly one function $\kk\colon [m]\to \N\setminus\{0\}$ (the empty function). Moreover by convention this function satisfies $\sum_{i=1}^m k_i = 0$ (as the summation runs over an empty set). Therefore, for $d=0$ and $m=0$ the subsum over $k_1,\ldots,k_m$ and $\ii$ above is equal to the free coefficient of the polynomial (which can be denoted by $c_\emptyset\ub{0}$), since the summation over $k_1,\ldots,k_m$ runs over a one-element set containing the empty index/function and for this index there is exactly one index $\ii\colon [m]\to \{1,\ldots,n\}$, which belongs to $[n]^{\underline{m}}$ (again the empty-index). Here we also use the convention that a product over an empty set is equal to one. On the other hand, for $d>0$, the contribution from $m=0$ is equal to zero (as the empty index $\kk$ does not satisfy the constraint $k_1+\ldots+k_m = d$ and so the summation over $k_1,\ldots,k_m$ runs over the empty set).

Using \eqref{eq_poly_rep_1} together with independence of $X_1,\ldots,X_n$, one may write
\begin{displaymath}
f(X) - \E f(X)
= \sum_{1\le d\le D} \sum_{m=1}^d \sum_{{k_1, \ldots, k_m > 0}\atop{k_1+\ldots+k_m=d}} \sum_{\ii\in[n]\uu{m}} c_{(i_1,k_1),\ldots,(i_m,k_m)}\ub{d} \sum_{\emptyset \neq J \subseteq [m]} \prod_{j\in J} (X_{i_j}^{k_j} - \E X_{i_j}^{k_j})\prod_{j\notin J} \E X_{i_j}^{k_j}.
\end{displaymath}
Rearranging the terms and using \eqref{eq:poly_symmetry} together with the triangle inequality, we obtain
\begin{displaymath}
|f(X) - \E f(X) |
\le \sum_{1\le d\le D} \sum_{a=1}^d \sum_{{k_1, \ldots, k_a > 0}\atop{k_1+\ldots+k_a=d}} \Big|\sum_{\ii \in [n]\uu{a}} d_{i_1,\ldots,i_a}\ub{k_1,\ldots,k_a} (X_{i_1}^{k_1} - \E X_{i_1}^{k_1})\cdots (X_{i_a}^{k_a} - \E X_{i_a}^{k_a})\Big|,
\end{displaymath}
where
\begin{align*}
d_{i_1,\ldots,i_a}\ub{k_1,\ldots,k_a} = \sum_{a\le m \le D}\sum_{{k_{a+1},\ldots,k_m > 0\colon}\atop{k_1+\ldots+k_m \le D}}\sum_{{i_{a+1},\ldots,i_m\colon}\atop{(i_1,\ldots,i_m)\in [n]\uu{m}}}\binom{m}{a}c_{(i_1,k_1),\ldots,(i_m,k_m)}\ub{k_1+\ldots+k_m}\E X_{i_{a+1}}^{k_{a+1}}\cdots \E X_{i_m}^{k_{i_m}}.
\end{align*}

Note that \eqref{eq:poly_symmetry} implies that for every permutation $\pi\colon [a]\to [a]$,
\begin{align}\label{eq:poly_symmetry_d}
d_{i_1,\ldots,i_a}\ub{k_1,\ldots,k_a} = d_{i_{\pi_1},\ldots,i_{\pi_a}}\ub{k_{\pi_1},\ldots,k_{\pi_a}}.
\end{align}
 Let now $X\ub{1},\ldots,X\ub{D}$ be independent copies of the random vector $X$ and $(\varepsilon_i\ub{j})_{i\le n,j\le D}$ an array of i.i.d. Rademacher variables independent of $(X\ub{j})_j$. For each $k_1,\ldots,k_a$, by decoupling inequalities (Theorem \ref{thm:decoupling} in the Appendix) applied to the functions
 \begin{displaymath}
 h_{i_1,\ldots,i_a}\ub{k_1,\ldots,k_a}(x_1,\ldots,x_a) =  d_{i_1,\ldots,i_a}\ub{k_1,\ldots,k_a}(x_1^{k_1} - \E X_{i_1}^{k_1})\cdots(x_a^{k_a} - \E X_{i_a}^{k_a})
 \end{displaymath}
 and standard symmetrization inequalities (applied conditionally $a$ times) we obtain,
\begin{align}\label{eq:after_decoupling}
&\|f(X) - \E f(X) \|_p\\
&\le C_D\sum_{d=1}^D \sum_{a=1}^d \sum_{{k_1, \ldots, k_a > 0}\atop{k_1+\ldots+k_a=d}} \bigg\|\sum_{\ii \in[n]\uu{a}} d_{i_1,\ldots,i_a}\ub{k_1,\ldots,k_a} \Big((X_{i_1}\ub{1})^{k_1} - \E (X_{i_1}\ub{1})^{k_1}\Big)\cdots \Big((X_{i_a}\ub{a})^{k_a} - \E (X_{i_a}\ub{a})^{k_a}\Big)\bigg\|_p\nonumber\\
&\le C_D\sum_{d=1}^D \sum_{a=1}^d \sum_{{k_1, \ldots, k_a > 0}\atop{k_1+\ldots+k_a=d}} \bigg\|\sum_{\ii \in[n]\uu{a}} d_{i_1,\ldots,i_a}\ub{k_1,\ldots,k_a} \Big(\varepsilon_{i_1}\ub{1}(X_{i_1}\ub{1})^{k_1} \cdots \varepsilon_{i_a}\ub{a}(X_{i_a}\ub{a})^{k_a} \Big)\bigg\|_p\nonumber
\end{align}
(note that in the first part of Theorem \ref{thm:decoupling} one does not impose any symmetry assumptions on the functions $h_\ii$).

We will now use the following standard comparison lemma (for reader's convenience its proof is presented in the Appendix).
\begin{lemma}\label{lemma_comp_Gauss} For any positive integer $k$, if $Y_1,\ldots,Y_n$ are independent symmetric  variables with $\|Y_i\|_{\psi_{2/k}} \le M$, then
\begin{displaymath}
\|\sum_{i=1}^n a_i Y_i\|_p \le C_k M\|\sum_{i=1}^n a_i g_{i1}\cdots g_{ik} \|_p,
\end{displaymath}
where $g_{ij}$ are i.i.d. $\mathcal{N}(0,1)$ variables.
\end{lemma}

Note that for any positive integer $k$ we have $\|X_i^k\|_{\psi_{2/k}} = \|X_i\|_{\psi_2}^k \le L^k$, so \eqref{eq:after_decoupling} together with the  above lemma (used repeatedly and conditionally) yield
\begin{multline}\label{eq_now_Gaussian}
\|f(X) - \E f(X) \|_p \\
\le C_D\sum_{1\le d\le D} L^d\sum_{a=1}^d \sum_{{k_1, \ldots, k_a > 0}\atop{k_1+\ldots+k_a=d}} \Big\|\sum_{\ii \in[n]\uu{a}} d_{i_1,\ldots,i_a}\ub{k_1,\ldots,k_a} (g\ub{1}_{i_1,1}\cdots g\ub{1}_{i_1,k_1})\cdots(g\ub{a}_{i_a,1}\cdots g\ub{a}_{i_a,k_a})\Big\|_p,
\end{multline}
where $(g_{i,k}\ub{j})$ is an array of i.i.d. standard Gaussian variables.
Consider now multi-indexed matrices $B_1,\ldots,B_D$ defined as follows. For $1\le d\le D$, and a multi-index $\rr = (r_1,\ldots,r_d)\in [n]^d$   let $\mathcal{I} = \{I_1,\ldots,I_a\}$ be the partition of $\{1,\ldots,d\}$ into the level sets of $\rr$ and $i_1,\ldots,i_a$ be the values corresponding to the level sets $I_1,\ldots,I_a$. Define moreover
\begin{align*}
b\ub{d}_{r_1,\ldots,r_d} = d\ub{\#I_1,\ldots,\#I_a}_{i_1,\ldots,i_a}
\end{align*}
 (note that thanks to \eqref{eq:poly_symmetry_d} this definition does not depend on the order of $I_1,\ldots,I_a$).
Finally, define the $d$-indexed matrix $B_d = (b\ub{d}_{\rr})_{\rr\in [n]^d}$.

Let us also define for $k_1,\ldots,k_a > 0$, $\sum_{i=1}^a k_i = d$ the partition $\mathcal{K}(k_1,\ldots,k_a) \in \mathcal{P}_d$ by splitting the set $\{1,\ldots,d\}$ into consecutive intervals of length $k_1,\ldots,k_a$, i.e.,
$\mathcal{K} = \{K_1,\ldots,K_a\}$, where for $l = 1,\ldots,a$, $K_l = \{1+\sum_{i=1}^{l-1} k_i,2+\sum_{i=1}^{l-1} k_i, \ldots,\sum_{i=1}^{l} k_i\}$.

 Applying Theorem \ref{thm_Latala_dec} to the right hand side of (\ref{eq_now_Gaussian}), we obtain
 \begin{align*}
&\|f(X) - \E f(X) \|_p \\
& \le C_D\sum_{1\le d\le D} L^d\sum_{a=1}^d  \sum_{{k_1, \ldots, k_a > 0}\atop{k_1+\ldots+k_a=d}}\Big\|\Big\langle B_d \circ \Ind{L(\mathcal{K}(k_1,\ldots,k_a))},\bigotimes_{j=1}^a \bigotimes_{k=1}^{k_j} (g\ub{j}_{i,k_j})_{i\le n}\Big\rangle \Big\|_p\nonumber \\
&\le C_D\sum_{1\le d\le D}L^d \sum_{a=1}^d \sum_{{k_1, \ldots, k_a > 0}\atop{k_1+\ldots+k_a=d}} \sum_{\mathcal{J} \in P_{d}} p^{\#\mathcal{J}/2}\|B_d \circ \Ind{L(\mathcal{K}(k_1,\ldots,k_a))}\|_\mathcal{J}.
 \end{align*}

 Note that for all $k_1,\ldots,k_a$ by Corollary \ref{cor_norm_monotonicity} we have $\|B_d\circ\Ind{L(\mathcal{K}(k_1,\ldots,k_a))}\|_\mathcal{J} \le C_d \|B_d\|_\mathcal{J}$. Thus  we obtain
 \begin{align*}
\|f(X) - \E f(X) \|_p\le C_D\sum_{1\le d\le D}  L^d\sum_{\mathcal{J} \in P_{d}} p^{\#\mathcal{J}/2}\|B_d \|_\mathcal{J}.
 \end{align*}

Our next goal is to replace $B_d$ in the above inequality by $\E \D^d f(X)$. To this end we will analyse the structure of the coefficients of $B_d$ and compare them with the integrated partial derivatives of $f$.

Let us first calculate $\E \D^d f(X)$. Consider $\rr\in [n]^d$, such that $i_1,\ldots,i_a$ are its distinct values, taken $l_1,\ldots,l_a$ times respectively. We have
\begin{multline*}
\E \frac{\partial^{d} f}{\partial x_{r_1}\cdots\partial x_{r_d}}(X) = \sum_{k_1\ge l_1,\ldots,k_a\ge l_a}\sum_{a\le m \le D} \sum_{{k_{a+1},\ldots,k_m > 0}\atop{k_1+\ldots+k_m \le D}} \sum_{{i_{a+1},\ldots,i_m}\atop{(i_1,\ldots,i_m) \in [n]\uu{m}}} \\
\Bigg[\binom{m}{a}a!c\ub{k_1+\ldots+k_m}_{(i_1,k_1),\ldots,(i_m,k_m)}\prod_{j=1}^a \E X_{i_j}^{k_j-l_j}\prod_{j=a+1}^m \E X_{i_j}^{k_j}\prod_{j=1}^a \frac{k_j!}{(k_j - l_j)!}\Bigg],
\end{multline*}
where we have used \eqref{eq:poly_symmetry}.

By comparing this with the definition of $b\ub{d}_{r_1,\ldots,r_d}$ and $d\ub{k_1,\ldots,k_a}_{i_1,\ldots,i_a}$ one can see that the sub-sum of the right hand side above corresponding to the choice $k_1 = l_1,\ldots,k_a = l_a$ is equal to $a!l_1!\cdots l_a! b\ub{d}_{r_1,\ldots,r_d}$.

In particular for $d=D$, since $l_1+\ldots+l_a = D$, we have
\begin{displaymath}
\E \frac{\partial^{D} f}{\partial x_{r_1}\cdots\partial x_{r_D}}(X) = a! l_1! \cdots l_a! b\ub{D}_{r_1,\ldots,r_D}
\end{displaymath}
and so
\begin{align*}
\|B_D\|_\mathcal{J}\le  \sum_{\mathcal{K}\in \mathcal{P}_D} \|B_D\circ \Ind{L(\mathcal{K})}\|_\mathcal{J} \le \sum_{\mathcal{K}\in \mathcal{P}_D} \|\D^D f(X)\circ \Ind{L(\mathcal{K})}\|_\mathcal{J} \le C_D \|\D^D f(X)\|_\mathcal{J},
\end{align*}
where in the last inequality we used Corollary \ref{cor_norm_monotonicity}.
Therefore if we prove that for all $d < D$ and all partitions $\mathcal{I} = \{I_1,\ldots,I_a\},\mathcal{J} = \{J_1,\ldots,J_b\} \in P_d$,
\begin{align}\label{eq_subgaussian_to_prove}
\|a!\#I_1!\cdots \#I_a! (B_d \circ \Ind{L(\mathcal{I})}) - \E \D^d f(X)\circ \Ind{L(\mathcal{I})}\|_\mathcal{J} \le C_D \sum_{d< k \le D} L^{k-d} \sum_{{\mathcal{K} \in P_k}\atop {\#\mathcal{K} = \#\mathcal{J}}}\|B_k\|_\mathcal{K},
\end{align}
then by simple reverse  induction (using again Corollary \ref{cor_norm_monotonicity}) we will obtain
\begin{align*}
\sum_{1\le d\le D} L^d \sum_{\mathcal{J} \in P_{d}} p^{\#\mathcal{J}/2}\|B_d \|_\mathcal{J} \le C_D\sum_{1\le d\le D} L^d\sum_{\mathcal{J} \in P_{d}} p^{\#\mathcal{J}/2}\|\E \D^d f(X)\|_\mathcal{J},
\end{align*}
which will end the proof of the theorem.

Fix any $d < D$ and partitions $\mathcal{I} = \{I_1,\ldots,I_a\}, \mathcal{J} = \{J_1,\ldots,J_b\}  \in P_d$. Denote  $l_i = \# I_i$.
 For every sequence $k_1,\ldots,k_a$ such that $k_i \ge l_i$ for $i\le a$ and there exists $i \le a$ such that $k_i > l_i$, let us define a $d$-indexed matrix $E\ub{d,k_1,\ldots,k_a}_\mathcal{I} = (e\ub{d,k_1,\ldots,k_a}_\rr)_{\rr \in [n]^d}$, such that $e\ub{d,k_1,\ldots,k_a}_\rr = 0$ if $\rr \notin L(\mathcal{I})$ and for $\rr \in L(\mathcal{I})$,
\begin{align*}
e\ub{d,k_1,\ldots,k_a}_\rr = \sum_{a\le m\le D}\sum_{{k_{a+1},\ldots,k_m > 0}\atop{k_1+\ldots+k_m \le D}} \sum_{{i_{a+1},\ldots,i_m}\atop{(i_1,\ldots,i_m) \in[n]^{\underline{m}}}} \binom{m}{a}c\ub{k_1+\ldots+k_m}_{(i_1,k_1),\ldots,(i_m,k_m)}\prod_{j=1}^a \E X_{i_j}^{k_j-l_j}\prod_{j=a+1}^m \E X_{i_j}^{k_j},
\end{align*}
where $i_1,\ldots,i_a$ are the values of $\rr$ corresponding to the level sets $I_1,\ldots,I_a$. We then have
\begin{align*}
\sum_{{k_1\ge l_1,\ldots,k_a \ge l_a}\atop{\exists_i k_i > l_i}} a!\frac{k_1!}{(k_1-l_1)!}\cdots \frac{k_a!}{(k_a-l_a)!}E_\mathcal{I}\ub{d,k_1,\ldots,k_a} = \E \D^d f(X)\circ \Ind{L(\mathcal{I})} - a!l_1!\cdots l_a! B_d\circ \Ind{L(\mathcal{I})}.
\end{align*}
Since we do not pay attention to constants depending only on $D$, by the above formula and the triangle inequality, to prove (\ref{eq_subgaussian_to_prove}) it is enough to show that for all sequences $k_1,\ldots,k_a$ such that $k_i \ge l_i$ for $i\le a$ and there exists $i \le a$ such that $k_i > l_i$ one has
\begin{align}\label{eq_subgaussian_reduced}
\|E\ub{d,k_1,\ldots,k_a}_\mathcal{I}\|_{\mathcal{J}} \le C_D L^{\sum_{j\le a} (k_j-l_j)}\|B_{k_1+\ldots+k_a}\|_\mathcal{K}
\end{align}
for some partition $\mathcal{K}\in \mathcal{P}_{k_1+\ldots+k_a}$ with $\# \mathcal{K} = \#\mathcal{J}$ (note that $\sum_{j\le a} l_j = d$). Therefore in what follows we will fix $k_1,\ldots,k_a$ as above and to simplify the notation we will write $E\ub{d}$ instead of $E\ub{d,k_1,\ldots,k_a}_\mathcal{I}$ and $e\ub{d}_\rr$ instead of $e\ub{d,k_1,\ldots,k_a}_\rr$.

Fix therefore any partition $\tilde{\mathcal{I}} = \{\tilde{I}_1,\ldots,\tilde{I_a}\} \in \mathcal{P}_{k_1+\ldots+k_a}$ such that $\#\tilde{I}_i = k_i$ and $I_i \subseteq \tilde{I}_i$ for all $i \le a$ (the specific choice of $\tilde{\mathcal{I}}$ is irrelevant). Finally define a $(k_1+\ldots+k_a)$-indexed matrix $\tilde{E}\ub{k_1+\ldots+k_a} = (\tilde{e}\ub{k_1+\ldots+k_a}_\rr)_{\rr\in [n]^d}$ by setting
\begin{align}\label{eq:E_tilde_construction}
\tilde{e}\ub{k_1+\ldots+k_a}_\rr
= e\ub{d}_{\rr_{[d]}} \ind{\rr\in L(\mathcal{\tilde{I}})}.
 \end{align}
 In other words the new matrix is created by embedding the $d$-indexed matrix into a ``generalized diagonal'' of a $(k_1+\ldots+k_a)$-indexed matrix by adding  $\sum_{j\le a} (k_j-l_j)$ new indices and assigning to them the values of old indices (for each $j \le a$ we add $k_j-l_j$ times the common value attained by $\rr_{\{1,\ldots,d\}}$ on $I_j$).

Recall now the definition of the coefficients $b\ub{d}_\rr$  and note that for any $\rr \in L(\mathcal{\tilde{I}})\subseteq [n]^{k_1+\ldots+k_a}$ we have $\tilde{e}\ub{k_1+\ldots+k_a}_\rr = b\ub{k_1+\ldots+k_a}_\rr\prod_{j=1}^a\E X_{i_j}^{k_j-l_j}$, where for $j\le a$, $i_j$ is the value of $\rr$ on its level set $\tilde{I}_j$.
This means that
$\tilde{E}\ub{k_1+\ldots+k_a} = (B_{k_1+\ldots+k_a}\circ \Ind{L(\mathcal{\tilde{I}})})\circ (\otimes_{s=1}^{k_1+\ldots+k_a} v_s)$,
where  $v_s = (\E X_i^{k_j-l_j})_{i\le n}$ if $s \in \{\min I_1,\ldots,\min I_a\}$ and $v_s = (1,\ldots,1)$ otherwise. Since $\|v_s\|_\infty \le (C_DL)^{k_j-l_j}$ if $s \in \{\min I_j\}_{j\le a}$ and $\|v_s\|_\infty = 1$ otherwise, by Lemma \ref{lem_tensor_product} this implies  that
for any $\mathcal{K} \in P_{k_1+\ldots+k_a}$,
\begin{align}\label{eq_subgaussian_aux}
\|\tilde{E}\ub{k_1+\ldots+k_a}\|_\mathcal{K} \le (C_D L)^{\sum_{j\le a}(k_j-l_j)}\|B_{k_1+\ldots+k_a}\circ \Ind{L(\tilde{\mathcal{I}})}\|_\mathcal{K} \le C_DL^{\sum_{j\le a}(k_j-l_j)}\|B_{k_1+\ldots+k_a}\|_\mathcal{K},
\end{align}
where in the last inequality we used Corollary \ref{cor_norm_monotonicity}.

We will now use the above inequality to prove \eqref{eq_subgaussian_reduced}. Consider the unique partition $\mathcal{K} = \{K_1,\ldots,K_b\}$ satisfying the following two conditions:
\begin{itemize}
\item for each $j\le b$, $J_j \subseteq K_j$,

 \item for each $s \in \{d+1,\ldots,k_1+\ldots+k_a\}$ if $s \in \tilde{I}_j$ and $\pi(s) := \min \tilde{I}_j\in J_k$, then $s \in K_k$. In other words all indices $s$, which in the construction of $\tilde{\mathcal{I}}$ were added to $I_j$ (i.e., elements of $\tilde{I}_j\setminus I_j$) are now added to the unique element of $\mathcal{J}$ containing $\pi(s) = \min \tilde{I}_j = \min I_j$.
\end{itemize}

Now, it is easy to see  that $\|E\ub{d}\|_\mathcal{J}\le \|\tilde{E}\ub{k_1+\ldots+k_a}\|_\mathcal{K}$.
Indeed, consider an arbitrary $x\ub{j} = (x_{\rr_{J_j}}\ub{j})_{|\rr_{J_j}|\le n}$, $j=1,\ldots,b$, satisfying $\|x\ub{j}\|_2\le 1$. Define $y\ub{j} = (y_{\rr_{K_j}}\ub{j})_{|\rr_{K_j}|\le n}$, $j =1,\ldots,b$ with the formula
\begin{displaymath}
y\ub{j}_{\rr_{K_j}} = x\ub{j}_{\rr_{K_j\cap[d]}}\prod_{s\in K_j \setminus [d]}\ind{r_s = r_{\pi(s)}}.
\end{displaymath}
We have $\|y\ub{j}\|_2 = \|x\ub{j}\|_2\le 1$. Moreover, by the construction of the matrix $\tilde{E}\ub{k_1+\ldots+k_a}$ (recall \eqref{eq:E_tilde_construction}), we have
\begin{displaymath}
\sum_{|\rr_{[d]}|\le n} e_{\rr_{[d]}}\ub{d} \prod_{j=1}^b x\ub{j}_{\rr_{J_j}} = \sum_{|\rr_{[k_1+\ldots+k_a]}|\le n} \tilde{e}_{\rr_{[k_1+\ldots+k_a]}}\ub{k_1+\ldots+k_a} \prod_{j=1}^b x\ub{j}_{\rr_{J_j}} = \sum_{|\rr_{[k_1+\ldots+k_a]}|\le n} \tilde{e}_{\rr_{[k_1+\ldots+k_a]}}\ub{k_1+\ldots+k_a} \prod_{j=1}^b y\ub{j}_{\rr_{K_j}}
\end{displaymath}
(in the last equality we used the fact that if $\rr \in L(\tilde{I})$, then for $s > d$, $r_{\pi(s)} = r_s$ and so $y\ub{j}_{\rr_{K_j}} = x\ub{j}_{\rr_{K_j\cap[d]}} = x\ub{j}_{\rr_{J_j}}$). By taking the supremum over $x\ub{j}$ one thus obtains $\|E\ub{d}\|_\mathcal{J}\le \|\tilde{E}\ub{k_1+\ldots+k_a}\|_\mathcal{K}$. Combining this inequality with (\ref{eq_subgaussian_aux}) proves (\ref{eq_subgaussian_reduced}) and thus (\ref{eq_subgaussian_to_prove}). This ends the proof of Theorem \ref{thm:subgaussian_intro}.

\subsection{Application: Subgraph counting in random graphs}

We will now apply results from Section \ref{sec:subgaussian} to some special cases of the problem of subgraph counting in Erd\H{o}s-R{\'e}nyi random graphs $G(n,p)$, which is often used as a test model for deviation inequalities  for polynomials in independent random variables. More specifically we will investigate the problem of counting cycles of fixed length.

It turns out that Theorem \ref{thm:subgaussian_intro} may give in some ranges of parameters optimal inequalities (leading to improvements of known results), whereas in some other regimes the estimates it gives are suboptimal.

\medskip
Let us first describe the setting (we will do it in a slightly more general form that needed for our example). We will consider undirected graphs $G = (V,E)$, where $V$ is a finite set of vertices and $E$ is the set of edges (i.e. two-element subsets of $V$). By $V_G = V(G)$ and $E_G = E(G)$ we mean the set of vertices and edges (respectively) of a graph $G$. Also, $v_G = v(G)$ and $e_G = e(G)$ denote the number of vertices and edges in $G$.
We say that a graph $H$ is a subgraph of a graph $G$ if $V_H \subseteq V_G$ and $E_H \subseteq E_G$ (thus a subgraph is non-necessarily induced). Graphs $H$ and $G$ are isomorphic if there is a bijection $\pi\colon V_H \to V_G$ such that for all distinct $v,w \in V_H$, $\{\pi(v),\pi(w)\} \in E_G$ iff $\{v,w\} \in E_H$.

For $p \in [0,1]$ consider now the Erd\H{o}s-R{\'e}nyi random graph $G = G(n,p)$, i.e., a graph with $n$ vertices (we will assume that $V_G = [n]$) whose edges are selected independently at random with probability $p$. In what follows we will be concerned with a number of copies of a given graph $H = ([k],E_H)$ in a graph $G$, i.e., the number of subgraphs of $G$ which are isomorphic to $H$. We will denote this random variable by $Y_H(n,p)$. To relate $Y_H(n,p)$ to polynomials, let us consider the family $C(n,2)$ of two-element subsets of $[n]$ and the family of independent random variables $X = (X_{e})_{e \in C(n,2)}$, such that $\p(X_{e} = 1) = 1 - \p(X_{e} = 0) = p$ (i.e., $X_{e}$ indicates whether the edge $e$ has been selected or not). Denote moreover by $\textup{Aut}(H)$ the group of isomorphisms of $H$ into itself and note that
\begin{displaymath}
Y_H(n,p) = \frac{1}{\#\textup{Aut}(H)} \sum_{\ii \in [n]\uu{k}} \prod_{{v,w \in [k]}\atop{v < w, \{v,w\} \in E(H)}} X_{\{i_v,i_w\}}.
\end{displaymath}

The right-hand side above is a homogeneous tetrahedral polynomial of degree $e_H$. Moreover the variables $X_{\{v,w\}}$ satisfy
\begin{displaymath}
\E \exp\Big(X_{\{v,w\}}^2\log(1/p)\Big) = 1 - p + p\cdot\frac{1}{p} \le 2
\end{displaymath}
and
\begin{displaymath}
\E \exp\Big(X_{\{v,w\}}^2\log 2\Big) \le 2,
\end{displaymath}
which implies that $\|X_{\{v,w\}}\|_{\psi_2} \le (\log(1/p))^{-1/2}\wedge (\log(2))^{-1/2} \le \sqrt{2} (\log(2/p))^{-1/2}$.

We can thus apply Theorem~\ref{thm:subgaussian_intro} to $Y_H(n,p)$ and obtain
\begin{align}\label{eq:general_graph}
\p\big(|Y_{H}(n,p) - \E Y_{H}(n,p) |\ge t\big)\le 2\exp\bigg(-\frac{1}{C_k}\min_{1 \le d\le k}\min_{\mathcal{J} \in P_d} \Big(\frac{t}{L_p^d\|\E\D^d f(X)\|_\mathcal{J}}\Big)^{2/\#\mathcal{J}}\bigg),
\end{align}
where $L_p = \sqrt{2} \big(\log(2/p)\big)^{-1/2}$ and  $f \colon \R^{C(n,2)} \to \R$ is given by
\begin{align*}
f((x_{e})_{e \in C(n,2)}) = \frac{1}{\# \textup{Aut}(H)} \sum_{\ii \in [n]\uu{k}} \prod_{{v,w \in [k]}\atop{v < w, \{v,w\} \in E}} x_{\{i_v,i_w\}}.
\end{align*}

Deviation inequalities for subgraph counts have been studied by many authors, to mention \cite{KimVuConc,JaRuInf,VuConc,JanOleRu,KimVUtr,ChatterjeeTr,DeMarcoKahnTr,DeMarcoKahnCl}. As it turns out the lower tail $\p(Y_H(n,p) \le \E Y_H(n,p) - t)$ is easier than the upper tail $\p(Y_H(n,p) \ge \E Y_H(n,p) +t)$. The lower tail turns out to be also lighter than the upper one. Since our inequalities concern $|Y_H(n,p) - \E Y_H(n,p)|$, we cannot hope to recover optimal lower tail estimates, however we can still hope to get bounds which in some range of parameters $n,p$ will agree with optimal upper tail estimates.

Of particular importance in literature is the law of large numbers regime, i.e., the case when $t = \varepsilon \E Y_H(n,p)$. In~\cite{JanOleRu} the Authors prove that for every $\varepsilon>0$ such that $\p\big(Y_H(n,p) \ge (1+\varepsilon)\E Y_H(n,p)\big) > 0$,
\begin{equation}\label{eq:JOR}
\exp\left(-C(H,\varepsilon)M_H^\ast(n,p)\log\frac{1}{p}\right)
  \le \p\big(Y_H(n,p) \ge (1+\varepsilon)\E Y_H(n,p)\big)
  \le \exp\big(-c(H,\varepsilon)M_H^\ast(n,p)\big)
\end{equation}
for certain constants $c(H,\varepsilon), C(H,\varepsilon)$ and a certain function $M_H^\ast(n,p)$. Since the general definition of $M_H^\ast$ is rather involved we will skip the details (in the examples considered in the sequel we will provide specific formulas). Note that if one disregards the constants depending only on $H$ and $\varepsilon$, the lower and upper estimate above differ by the factor $\log(1/p)$ in the exponent. To our best knowledge providing a lower and upper bound  for general $H$, which would agree up to multiplicative constants in the exponent (depending only on $H$ and $\varepsilon$, but not on $n$ or $p$) is an open problem.

We will now specialize to the case when $H$ is a cycle. For simplicity we will first present the case of the triangle $K_3$ (the clique with three vertices). For this graph the upper bound from \cite{JanOleRu} has been recently strengthened to match the lower one (up to a constant depending only on $\varepsilon$) by Chatterjee \cite{ChatterjeeTr} and DeMarco and Kahn \cite{DeMarcoKahnTr} (who also obtained a similar result for general cliques \cite{DeMarcoKahnCl}). In the next section we show that if $p$ is not too small, the inequality \eqref{eq:general_graph} also allows to recover the optimal upper bound. In Section \ref{sec:cycles} we provide an upper bound for cycles of arbitrary (fixed) length $k$, which is optimal for $p \ge n^{-\frac{k-2}{2(k-1)}} \log^{-\frac12} n$.

\subsubsection{Counting triangles}
Assume that $H = K_3$ and let us analyse the behaviour of $\|\E \D^d f(X)\|_\mathcal{J}$ for
$d = 1,2,3$. Of course in this case $\#\textup{Aut}(H) = 6$.

We have for any $e = \{v,w\}$, $v,w \in [n]$,
\begin{displaymath}
\frac{\partial}{\partial x_e} f(x) = \sum_{i \in [n]\setminus\{v,w\}} x_{\{i,v\}} x_{\{i,w\}}
\end{displaymath}
and so $\|\E \D f(X)\|_{\{1\}}= (n-2)p^2 \sqrt{n(n-1)/2} \le  n^2p^2$.

For $e_1 = e_2$ or when $e_1$ and $e_2$ do not have a common vertex, we have $\frac{\partial^2}{\partial x_{e_1}\partial x_{e_2}} f = 0$, whereas for $e_1,e_2$ sharing exactly one vertex, we have
\begin{displaymath}
\frac{\partial^2}{\partial x_{e_1}\partial x_{e_2}} f (x) = x_{\{v,w\}},
\end{displaymath}
where $v,w$ are the vertices of $e_1,e_2$ distinct from the common one. Therefore
\begin{displaymath}
\E \D^2 f(X) = p (\ind{\textrm{$e_1,e_2$ have exactly one common vertex}})_{e_1,e_2 \in C(n,2)}.
\end{displaymath}
Using the fact that $\E \D^2 f(X)$ is symmetric and for each $e_1$ the sum of entries of $\E \D^2 f(X)$ in the row corresponding to $e_1$ equals $2p(n-2)$, we obtain $\|\E \D^2 f(X)\|_{\{1\}\{2\}} = 2p(n-2) \le 2pn$. One can also easily see that $\|\E \D^2 f(X)\|_{\{1,2\}} = p\sqrt{n(n-1)(n-2)} \le pn^{3/2}$.

Finally
\begin{displaymath}
\frac{\partial^3}{\partial x_{e_1}\partial x_{e_2} \partial x_{e_3}} f = \ind{\textrm{$e_1,e_2,e_3$ form a triangle}}
\end{displaymath}
and thus $\|\E \D^3f(X)\|_{\{1,2,3\}} = \sqrt{n(n-1)(n-2)} \le n^{3/2}$. Moreover, due to symmetry we have
\begin{displaymath}
\|\E \D^3 f(X)\|_{\{1,2\}\{3\}} = \|\E \D^3 f(X)\|_{\{1,3\}\{2\}} = \|\E \D^3 f(X)\|_{\{2,3\}\{1\}}.
\end{displaymath}
Consider arbitrary $(x_{e_1})_{e_1 \in C(n,2)}$ and $(y_{e_2,e_3})_{e_2,e_3 \in C(n,2)}$ of norm one. We have
\begin{align*}
&\sum_{e_1,e_2,e_3} \ind{\textrm{$e_1,e_2,e_3$ form a triangle}}x_{e_1}y_{e_2,e_3} \le
 \sqrt{\sum_{e_1} \Big( \sum_{e_2,e_3}\ind{\textrm{$e_1,e_2,e_3$ form a triangle}}y_{e_2,e_3} \Big)^2}\\
&\le \sqrt{\sum_{e_1} \Big( \sum_{e_2,e_3}\ind{\textrm{$e_1,e_2,e_3$ form a triangle}} \Big) \Big( \sum_{e_2,e_3}\ind{\textrm{$e_1,e_2,e_3$ form a triangle}}y_{e_2,e_3}^2 \Big) }\\
&= \sqrt{2(n-2)} \sqrt{ \sum_{e_2,e_3} y_{e_2,e_3}^2 \sum_{e_1} \ind{\textrm{$e_1,e_2,e_3$ form a triangle}} }
 \le \sqrt{2(n-2)},
\end{align*}
where the first two inequalities follow by the Cauchy-Schwarz inequality and the last one from the fact that for each $e_2,e_3$ there is at most one $e_1$ such that $e_1,e_2,e_3$ form a triangle. We have thus obtained $\|\E \D^3 f(X)\|_{\{1,2\}\{3\}} = \|\E \D^3 f(X)\|_{\{1,3\}\{2\}} = \|\E \D^3 f(X)\|_{\{2,3\}\{1\}} \le \sqrt{2n}$.

It remains to estimate $\|\E \D^3 f(X)\|_{\{1\}\{2\}\{3\}}$. For all $(x_e)_{e\in C(n,2)}$, $(y_e)_{e\in C(n,2)}$, $(z_e)_{e\in C(n,2)}$ of norm one we have by the Cauchy-Schwarz inequality
\begin{align*}
&\sum_{e_1,e_2,e_3}\ind{\textrm{$e_1,e_2,e_3$ form a triangle}}x_{e_1}y_{e_2} z_{e_3} = \sum_{(i_1,i_2,i_3) \in [n]\uu{3}}x_{\{i_1,i_2\}}y_{\{i_2,i_3\}}z_{\{i_1,i_3\}}\\
&\le \sum_{i_1 \in [n]} \Big(\sum_{(i_2,i_3) \in ([n]\setminus{\{i_1\}})\uu{2}} x_{\{i_1,i_2\}}^2 z_{\{i_1,i_3\}}^2\Big)^{1/2}\Big(\sum_{(i_2,i_3) \in ([n]\setminus{\{i_1\}})\uu{2}} y_{\{i_2,i_3\}}^2\Big)^{1/2}\\
&\le \sqrt{2}\sum_{i_1 \in [n]}\Big(\sum_{i_2\in [n]\setminus{\{i_1\}}} x_{\{i_1,i_2\}}^2\Big)^{1/2}\Big(\sum_{i_3\in [n]\setminus{\{i_1\}}} z_{\{i_1,i_3\}}^2\Big)^{1/2}\\
&\le \sqrt{2}\Big(\sum_{(i_1,i_2)\in [n]\uu{2}} x_{\{i_1,i_2\}}^2\Big)^{1/2}\Big(\sum_{(i_1,i_3)\in [n]\uu{2}} z_{\{i_1,i_3\}}^2\Big)^{1/2}
\le 2^{3/2},
\end{align*}
which gives $\|\E \D^3 f(X)\|_{\{1\}\{2\}\{3\}} \le 2^{3/2}$.

Using \eqref{eq:general_graph} together with the above estimates, we obtain
\begin{prop} \label{prop:triangle}For any $t >0$,
\begin{multline*}
\p\big(|Y_{K_3}(n,p) - \E Y_{K_3}(n,p) | \ge t\big)\\
\le 2\exp\Big(-\frac{1}{C}\min\Big(\frac{t^2}{L_p^6 n^3 + L_p^4p^2n^3 + L_p^2p^4 n^4},\frac{t}{L_p^3n^{1/2} +L_p^2 p n },\frac{t^{2/3}}{L_p^2}\Big)\Big),
\end{multline*}
where $L_p = \big(\log(2/p)\big)^{-1/2}$.
\end{prop}

In particular for $t = \varepsilon \E Y_{K_3}(n,p) = \varepsilon \binom{n}{3} p^3$,
\begin{multline*}
\p\big(|Y_{K_3}(n,p) - \E Y_{K_3}(n,p) | \ge \varepsilon \E Y_{K_3}(n,p) \big)
\\ \le 2\exp\Big(-\frac{1}{C} \min\Big(\varepsilon^2 n^3p^6 \log^3(2/p), (\varepsilon^2 \land \varepsilon^{2/3}) n^2p^2\log(2/p)\Big)\Big).
\end{multline*}
Thus for $p \ge n^{-\frac14}\log^{-\frac12} n$ we obtain
\begin{displaymath}
\p\big(|Y_{K_3}(n,p) - \E Y_{K_3}(n,p) | \ge \varepsilon \E Y_{K_3}(n,p) \big) \le 2\exp\big(-(\varepsilon^2 \land \varepsilon^{2/3}) n^2p^2\log(2/p)\big).
\end{displaymath}
By Corollary 1.7 in \cite{JanOleRu}, if $p \ge 1/n$, then $\frac{1}{C} n^2p^2 \le M^\ast_{K_3}(n,p) \le C n^2p^2$ (recall \eqref{eq:JOR}) and so for $p \ge n^{-1/4} \log^{-1/2} n$ the estimate obtained from the above proposition is optimal. As already mentioned the optimal estimate has been recently obtained  in the full range of $p$ by Chatterjee, DeMarco and Kahn. Unfortunately it seems that using our general approach we are not able to recover the full strength of their result. From Proposition \ref{prop:triangle} one can also see that Theorem \ref{thm:subgaussian_intro}, when specialized to polynomials in $0$-$1$ random variables is not directly comparable with the family of Kim-Vu inequalities. As shown in \cite{JaRuInf} (see table 2 therein), various inequalities by Kim and Vu give for the triangle counting problem exponents $-\min(n^{1/3}p^{1/6}, n^{1/2}p^{1/2})$, $-n^{3/2}p^{3/2}$, $-np$ (disregarding logarithmic factors). Thus for ``large'' $p$ our inequality performs better than those by Kim-Vu, whereas for ``small'' $p$ this is not the case (note that the Kim-Vu inequalities give meaningful bounds for $p \ge C n^{-1}$ while ours only for $p \ge C n^{-1/2}$). As already mentioned in the introduction the fact that our inequalities degenerate for small $p$ is not surprising as even for sums of independent $0$-$1$ random variables, when $p$ becomes small, general inequalities for the sums of independent random variables with sub-Gaussian tails do not recover the correct tail behaviour (the $\|\cdot\|_{\psi_2}$ norm of the summands becomes much larger than the variance).

\subsubsection{Counting cycles} \label{sec:cycles}

We will now generalize Proposition \ref{prop:triangle} to cycles of arbitrary length. If $H$ is a cycle of length $k$, then by Corollary 1.7 in~\cite{JanOleRu}, $\frac{1}{C} n^2p^2 \le M^\ast_{H}(n,p) \le C n^2p^2$ for $p\ge 1/n$. Thus the bounds for the upper tail from~\eqref{eq:JOR} imply that for $p \ge 1/n$,
\begin{displaymath}
\exp\big(-C(k,\varepsilon)n^2p^2\log(1/p)\big) \le \p\big(Y_H(n,p) \ge (1+\varepsilon)\E Y_H(n,p) \big) \le \exp\big(-c(k,\varepsilon)n^2p^2\big)
\end{displaymath}
for every $\varepsilon > 0$ for which the above probability is not zero.

We will show that similarly as for triangles, Theorem~\ref{thm:subgaussian_intro} allows to strengthen the upper bound if $p$ is not too small with respect to $n$. More precisely, we have the following
\begin{prop}\label{prop:cycles} Let $H$ be a cycle of length $k$. Then for every $t > 0$,
\begin{displaymath}
\p\big(|Y_H(n,p) - \E Y_H(n,p)| \ge t\big)
\le 2\exp\Big(-\frac{1}{C_k} \Big(\frac{t^2}{L_p^{2k}n^k} \wedge \min_{\substack{1\le l\le d\le k \colon\\ d<k\;\textup{or}\; l>1}}\Big(\frac{t^{2/l}}{L_p^{2d/l} p^{2(k-d)/l}n^{(2k-d-l)/l}}\Big)\Big)\Big),
\end{displaymath}
where $L_p = \big(\log(2/p)\big)^{-1/2}$. In particular for every $\varepsilon > 0$ and $p \ge n^{-\frac{k-2}{2(k-1)}}\log^{-1/2} n$,
\begin{displaymath}
\p\big(Y_H(n,p) \ge (1+\varepsilon)\E Y_H(n,p)\big) \le 2\exp\Big(-\frac{1}{C_k} (\varepsilon^2 \land \varepsilon^{2/k}) n^2p^2\log(2/p)\Big).
\end{displaymath}
\end{prop}

To prove the above proposition we need to estimate the corresponding $\|\cdot\|_\mathcal{J}$ norms. Since a major part of the argument does not rely on the fact that $H$ is a cycle and bounds on $\|\cdot\|_\mathcal{J}$ norms may be of independent interest, we will now consider arbitrary graphs. Let thus $H$ be a fixed graph with no isolated vertices.

Similarly to~\cite{JanOleRu}, it will be more convenient to count ``ordered'' copies of a graph $H$ in $G(n,p)$. Namely, for $H = ([k], E_H)$, each sequence of $k$ distinct vertices in the clique $K_n$, $\ii \in [n]^{\underline{k}}$ determines an ordered copy $G_{\ii}$ of $H$ in $K_n$, where $G_\ii = \ii(H)$, i.e., $V(G_\ii) = \ii([k])$ and $E(G_\ii) = \{ \ii(e) \colon e \in E(H) \} = \left\{ \{i_u, i_v\} \colon \{u, v\} \in E(H) \right\}$. Define
\[
  X_H(n,p) := \sum_{\ii \in [n]^{\underline{k}}} \ind{G_{\ii} \subseteq G(n,p)}
                = \sum_{\ii \in [n]^{\underline{k}}} \; \prod_{\tilde{e} \in E(G_\ii)} X_{\tilde{e}}.
\]
Clearly $X_H(n,p) = \# \textup{Aut}(H) Y_H(n,p)$ and $X_H(n,p) = f(X)$ where
\begin{align}\label{eq:counting_function}
  f(x) := \sum_{\ii \in [n]^{\underline{k}}} \; \prod_{\tilde{e} \in E(G_{\ii})} x_{\tilde{e}}
        = \sum_{\ii \in [n]^{\underline{k}}} \; \prod_{e \in E(H)} x_{\ii(e)}.
\end{align}
A sequence of distinct edges $(\tilde{e}_1, \ldots, \tilde{e}_d) \in E(K_n)^{\underline{d}}$ determines a subgraph $G_0 \subseteq K_n$ with $V(G_0) = \bigcup_{i=1}^d \tilde{e}_i$, $E(G_0) = \{\tilde{e}_1, \ldots, \tilde{e}_d\}$. Note that
\[
  \partial_{G_0} f(x) := \frac{\partial^d f(x)}{\partial x_{\tilde{e}_1} \cdots \partial x_{\tilde{e}_d}} =
    \sum_{\ii \in [n]^{\underline{k}} \colon G_{\ii} \supseteq G_0} \;
       \prod_{\tilde{e} \in E(G_{\ii}) \setminus E(G_0)} x_{\tilde{e}}
\]
and thus
\[
  \E \partial_{G_0} f(X) = p^{e(H) - d} \#\{ \ii \in [n]^{\underline{k}} \colon G_0 \subseteq G_\ii\}
\]
Consider $\ee = (e_1, \ldots, e_d) \in E(H)^{\underline{d}}$ and let $H_0(\ee)$ be the subgraph of $H$ with $V(H_0(\ee)) = \bigcup_{i=1}^d e_i$, $E(H_0(\ee)) = \{e_1, \ldots, e_d\}$. Clearly, for any $\ii \in [n]^{\underline{k}}$, $\ii(H_0(\ee)) \subseteq G_\ii$. We write $(e_1, \ldots e_d) \simeq (\tilde{e}_1, \ldots, \tilde{e}_d)$ if there exists $\ii \in [n]^{\underline{k}}$ such that $\ii(e_j) = \tilde{e}_j$ for $j = 1, \ldots, d$.
Note that given $(\tilde{e}_1, \ldots, \tilde{e}_d) \in E(K_n)^{\underline{d}}$ and the corresponding graph $G_0$,
\begin{align*}
  \#\{ \ii \in [n]^{\underline{k}} \colon G_0 \subseteq G_\ii\}
&= \sum_{\ee \in E(H)^{\underline{d}}} \#\{ \ii \in [n]^{\underline{k}} \colon \ii(e_j) = \tilde{e}_j \text{ for $j = 1, \ldots, d$}\} \\
 &= \sum_{\ee \in E(H)^{\underline{d}} } 2^{s(H_0(\ee))} (n - v(H_0(\ee)))^{\underline{k - v(H_0(\ee))}} \ind{(\tilde{e}_1, \ldots, \tilde{e}_d) \simeq \ee},
\end{align*}
where for a graph $G$, $v(G)$ is the number of vertices of $G$ and $s(G)$ is the number of edges in $G$ with no other adjacent edge. Therefore,
\[
  \E \D^d f(X) = p^{e(H) - d} \sum_{\ee \in E(H)^{\underline{d}} } 2^{s(H_0(\ee))} (n - v(H_0(\ee)))^{\underline{k - v_{H_0(\ee)}}} \left( \ind{(\tilde{e}_1, \ldots, \tilde{e}_d) \simeq \ee} \right)_{(\tilde{e}_1 \ldots, \tilde{e}_d)}.
\]
Let $\mathcal{J}$ be a partition of $[d]$. By the triangle inequality for the norms $\norm{\cdot}_{\mathcal{J}}$,
\begin{equation}\label{norms-for-subgraph-count}
  \norm{\E \D^d f(X)}_{\mathcal{J}} \le p^{e(H) - d} \sum_{\ee\in E(H)^{\underline{d}}} 2^{s(H_0(\ee))} n^{k - v(H_0(\ee))} \norm{\left( \ind{(\tilde{e}_1, \ldots, \tilde{e}_d) \simeq \ee} \right)_{(\tilde{e}_1 \ldots, \tilde{e}_d)}}_\mathcal{J}.
\end{equation}
The norms appearing on the right hand side of~(\ref{norms-for-subgraph-count}) are handled by the following
\begin{lemma}\label{lemma:norms-for-cycles}
  Fix $1 \le d \le e(H)$, $\ee = (e_1, \ldots, e_d) \in E(H)^{\underline{d}}$ and $\mathcal{J} = \{J_1,\ldots,J_l\} \in P_d$. Let $H_0 = H_0(\ee)$ and for $r = 1, \ldots, l$, let $H_r$ be a subgraph of $H_0$ spanned by the set of edges $\{e_j \colon j \in J_r\}$. Then,
\begin{multline*}
  \norm{\left( \ind{(\tilde{e}_1, \ldots, \tilde{e}_d) \simeq (e_1, \ldots, e_d)} \right)_{(\tilde{e}_1 \ldots, \tilde{e}_d)}}_{\mathcal{J}} \le
  2^{-s(H_0) + \frac12 \sum_{r=1}^l s(H_r)} \\
\times n^{\frac12 \# \{ v \in V(H_0) \colon \textup{$v \in V(H_r)$ for exactly one $r \in [l]$} \}}.
\end{multline*}
\end{lemma}
\begin{proof}
We shall bound the sum
\begin{equation}\label{sum-defining-the-norm-for-subgraph-count}
  \sum_{\tilde{e}_1, \ldots, \tilde{e}_d \in E(K_n)} \ind{(\tilde{e}_1, \ldots, \tilde{e}_d) \simeq \ee} \prod_{r=1}^l x_{(\tilde{e}_j)_{j \in J_r}}^{(r)}
\end{equation}
under the constraints $\sum_{(\tilde{e}_j)_{j \in J_r} \in E(K_n)^{J_r}} \left(x_{(\tilde{e}_j)_{j \in J_r}}^{(r)}\right)^2 \le 1$ for $r = 1, \ldots, l$. Note that we can assume $x^{(r)} \ge 0$ for all $r \in [l]$. Rewrite the sum~(\ref{sum-defining-the-norm-for-subgraph-count}) as the sum over a sequence of vertices instead of edges:
\[
  2^{-s(H_0)} \sum_{\ii \in [n]^{\underline{V(H_0)}}} \; \prod_{r=1}^l x_{(\ii(e_j))_{j \in J_r}}^{(r)},
\]
where for two sets $A,B$, $A^{\underline{B}}$ is the set of 1-1 functions from $B$ to $A$.
Further note that it is enough to prove the desired bound for the sum
\begin{equation}\label{sum2-for-the-norm-for-subgraph-count}
  2^{-s(H_0)} \sum_{\ii \in [n]^{\underline{V(H_0)}}} \; \prod_{r=1}^l y_{\ii_{V(H_r)}}^{(r)}
\end{equation}
under the constraints $2^{-s(H_r)} \sum_{\ii \in [n]^{\underline{V(H_r)}}} \left( y_{\ii_{V(H_r)}}^{(r)} \right)^2 \le 1$ for each $r = 1, \ldots, l$. Indeed, given $x$'s, for each $r = 1, \ldots, l$ and all $\ii \in [n]^{\underline{V(H_r)}}$ take
$y_{\ii_{V(H_r)}}^{(r)} = x_{(\ii(e_j)_{j \in J_r})}^{(r)}$ and notice that the sum~(\ref{sum2-for-the-norm-for-subgraph-count}) equals the sum~(\ref{sum-defining-the-norm-for-subgraph-count}) while the constraints for $x$'s imply the constraints for $y$'s. Finally, by homogeneity and the fact that the sum \eqref{sum2-for-the-norm-for-subgraph-count} does not depend on the full graph structure but only on the sets of vertices of the graphs $H_r$, the lemma will follow from the statement: For a sequence of finite, non-empty sets $V_1, \ldots, V_l$, let $V = V_1 \cup \ldots \cup V_l$. Then
\begin{equation}\label{sum3-for-the-norm-for-subgraph-count}
  \sum_{\ii \in [n]^{\underline{V}}} \; \prod_{r=1}^l y_{\ii_{V_r}}^{(r)}
  \le n^{\frac12 \# \{ v \in V \colon \text{$v \in V_r$ for exactly one $r \in [l]$} \}}
\end{equation}
for $y^{(1)}, \ldots, y^{(l)} \ge 0$ satisfying
\begin{equation}\label{constraint-for-then-norm-for-subgraph-count}
  \sum_{\ii \in [n]^{\underline{V_r}}} \left( y_{\ii_{V_r}}^{(r)} \right)^2 \le 1.
\end{equation}

We prove (\ref{sum3-for-the-norm-for-subgraph-count}) by induction on $\# V$. For $V = \emptyset$ (and $l=0$), (\ref{sum3-for-the-norm-for-subgraph-count}) holds trivially. For the induction step fix any $v_0 \in V$ and put $R = \{r \in [l] \colon v_0 \in V_r\}$. We write
\[
  \sum_{\ii \in [n]^{\underline{V}}}\;\prod_{r=1}^l y_{\ii_{V_r}}^{(r)}
=
  \sum_{\ii \in [n]^{\underline{V \setminus \{v_0\}}}} \left( \left(\prod_{r \in [l] \setminus R} y_{\ii_{V_r}}^{(r)} \right)
\sum_{i_{v_0} \in [n] \setminus \ii(V \setminus \{v_0\})} \; \prod_{r \in R} y_{\ii_{V_r}}^{(r)} \right).
\]
We bound the inner sum using the Cauchy-Schwarz inequality. If $\# R \ge 2$, we get
\[
  \sum_{i_{v_0} \in [n] \setminus \ii(V \setminus \{v_0\})} \; \prod_{r \in R} y_{\ii_{V_r}}^{(r)}
\le \prod_{r \in R} \left( \sum_{i_{v_0} \in [n] \setminus \ii(V \setminus \{v_0\})}  \left( y_{\ii_{V_r}}^{(r)} \right)^2 \right)^{1/2},
\]
and if $R = \{r_0\}$ then
\[
  \sum_{i_{v_0} \in [n] \setminus \ii(V \setminus \{v_0\})}  y_{\ii_{V_{r_0}}}^{(r_0)}
\le \sqrt{n} \left( \sum_{i_{v_0} \in [n] \setminus \ii(V \setminus \{v_0\})}  \left( y_{\ii_{V_{r_0}}}^{(r_0)} \right)^2 \right)^{1/2}.
\]
Now, for each $r \in R$ put $W_r = V_r \setminus \{v_0\}$ and define
\[
  z_{\ii_{W_r}}^{(r)} = \left( \sum_{i_{v_0} \in [n] \setminus \ii(W_r )} \left( y_{\ii_{V_r}}^{(r)} \right)^2 \right)^{1/2} \text{ for all $\ii_{W_r} \in [n]^{\underline{W_r}}$}.
\]
Note that if $W_r = \emptyset$ then $z^{(r)}$ is a scalar and by~(\ref{constraint-for-then-norm-for-subgraph-count}), $0 \le z^{(r)} \le 1$.
For $r \in [l] \setminus R$, just put $W_r = V_r$ and $z^{(r)} \equiv y^{(r)}$. Let $L = \{ r \in [l] \colon W_r \neq \emptyset\}$. Combining the estimates obtained above, we arrive at
\[
  \sum_{\ii \in [n]^{\underline{V}}} \; \prod_{r=1}^l y_{\ii_{V_r}}^{(r)} \le (\sqrt{n})^{\ind{\text{$v_0 \in V_r$ for exactly one $r \in [l]$}}}
  \sum_{\ii \in [n]^{\underline{V \setminus \{v_0\}}}} \;
    \prod_{r \in L} z_{\ii_{W_r}}^{(r)}.
\]
Now we use the induction hypothesis for the sequence of sets $(W_r)_{r \in L}$ and the vectors $z^{(r)}$, $r \in L$ (note that $\sum_{\ii\in[n]^{\underline{W_r}}}(z_{\ii_{W_r}}\ub{r})^2 \le 1$).
\end{proof}
\paragraph{Remark}
  The bound in Lemma~\ref{lemma:norms-for-cycles} is essentially optimal, at least for large $n$, say $n \ge 2 k$. To see this let us analyse optimality of~(\ref{sum3-for-the-norm-for-subgraph-count}) under the constraints~(\ref{constraint-for-then-norm-for-subgraph-count}) (it is easy to see that this is equivalent to the optimality in the original problem). Denote $V_0 = \{ v \in V \colon v \in V_r \text{ for exactly one $r \in [l]$}\}$. Fix any $\ii^{(0)} \in [n]^{\underline{k}}$. Then for $r =1, \ldots, l$ take
\[
  y_{\ii_{V_r}}^{(r)} = \begin{cases}
    n^{-\frac12 \#(V_r \cap V_0)} & \text{if $\ii_{V_r \setminus V_0} \equiv \ii_{V_r \setminus V_0}^{(0)}$} \\
    0 & \text{otherwise.}
  \end{cases}
\]
The vectors $y\ub{r}$ satisfy the constraints~(\ref{constraint-for-then-norm-for-subgraph-count}) and
\[ \begin{split}
  \sum_{\ii \in [n]^{\underline{V}}} \; \prod_{r=1}^l y_{\ii_{V_r}}^{(r)}
 &= \sum_{\ii \in [n]^{\underline{V}} \colon \ii_{V \setminus V_0} \equiv \ii_{V \setminus V_0}^{(0)}} \prod_{r=1}^l n^{-\frac12 \#(V_r \cap V_0)} \\[1ex]
     &= \left(n - \#(V \setminus V_0)\right)^{\underline{\# V_0}} \, n^{-\frac12 \# V_0} \ge (n/2)^{\# V_0} n^{-\frac12 \# V_0} = 2^{-\# V_0} n^{\frac12 \# V_0}.
\end{split} \]

\medskip
Combining Lemma \ref{lemma:norms-for-cycles} with \eqref{norms-for-subgraph-count} we obtain
\begin{lemma} \label{le:norm_estimates_cycles}
Let $H$ be any graph with $k$ vertices, which are not isolated, and let $f$ be defined by \eqref{eq:counting_function}. Then for any $1\le d\le e(H)$ and any $\mathcal{J} = \{J_1,\ldots,J_l\} \in P_d$,
\begin{multline*}
\|\E \D^d f(X)\|_\mathcal{J} \\
\le p^{e(H)-d}\sum_{\ee\in E(H)^{\underline{d}}}2^{\frac{1}{2}\sum_{r=1}^l s(H_r(\ee))} n^{k - v(H_0(\ee)) + \frac{1}{2}\#\{v\in V(H_0(\ee))\colon \textup{$v \in V(H_r(\ee))$ for exactly one $r\in[l]$}\}},
\end{multline*}
where for $\ee \in E(H)^{\underline{d}}$ and $r \in [l]$, $H_r(\ee)$ is the subgraph of $H_0(\ee)$ spanned by $\{e_j\colon j\in J_r\}$.
\end{lemma}

We are now ready for
\begin{proof}[Proof of Proposition \ref{prop:cycles}]
We will use Lemma \ref{le:norm_estimates_cycles} to estimate $\|\E \D^d f(X)\|_\mathcal{J}$ for any $d \le k$ and $\mathcal{J} \in P_d$ with $\#\mathcal{J} = l$. Note that for any $\ee\in \E(H)^{\underline{d}}$,
\begin{multline*}
  v(H_0(\ee)) - \frac12 \# \{ v \in V(H_0(\ee)) \colon \text{$v \in V(H_r(\ee))$ for exactly one $r \in [l]$}\} \\[1ex]
  = \frac12 \big( v(H_0(\ee)) + \#\{ v \in V(H_0(\ee)) \colon \text{$v$ belongs to more than one $V(H_r(\ee))$}\} \big) \\[1ex]
  \begin{cases} = k/2 & \text{if $d = k$ and $l=1$,} \\[1ex]
                                  \ge \frac12 (d+l) & \text{otherwise},
                    \end{cases}
\end{multline*}
where to get the second inequality we used the fact that each vertex of $H$ has degree two and the inclusion-exclusion formula. Thus we obtain
\begin{align*}
  \|\E\D^k f(X)\|_{\{[k]\}} &\le n^{k/2}, \\
  \|\E\D^d f(X)\|_\mathcal{J} &\le C_k p^{k-d}n^{k - \frac{1}{2}d - \frac{1}{2}l} \quad \text{if $d < k$ or $l > 1$}.
\end{align*}
Together with~\eqref{eq:general_graph} this yields the first inequality of the proposition. Using the fact that $\E Y_H(n,p) \ge \frac{1}{C_k} n^k p^k$, the second inequality follows by simple calculations.
\end{proof}

\section{Refined inequalities for polynomials in independent random variables satisfying the modified log-Sobolev inequality}\label{sec:Weibull}
In this section we refine the inequalities which can be obtained from Theorem \ref{thm:main} for polynomials in independent random variables satisfying the $\beta$-modified log-Sobolev inequality \eqref{eq:modifiedLS} with $\beta > 2$. To this end we will use Theorem \ref{thm:AidaStroock} together with a result from \cite{Chaos3d}, which is a counterpart of Theorem \ref{thm_Latala_dec} for homogeneous tetrahedral polynomials in general independent symmetric random variables with log-concave tails, however only of degree at most 3. We recall that for a set $I$, by $P_I$ we denote the family of partitions of $I$ into pairwise disjoint, nonempty sets.

Theorems 3.1 and 3.2 and 3.4 from \cite{Chaos3d} specialized to Weibull variables can be translated into

\begin{theorem}\label{thm:Weibull_chaos} Let $\alpha \in [1,2]$ and let $Y_1,\ldots,Y_n$ be a sequence of i.i.d. symmetric random variables satisfying $\p(|Y_i| \ge t) = \exp(-t^\alpha)$. Define $Y = (Y_1,\ldots,Y_n)$ and let $Z_1,\ldots,Z_d$ be independent copies of $Y$.
Consider a $d$-indexed matrix $A$. Define also
\begin{equation}\label{eq:mdpA}
m_d(p, A) = \sum_{I\subseteq [d]} \sum_{\mathcal{J} \in P_I}\sum_{\mathcal{K}\in P_{[d]\setminus I}} p^{\#\mathcal{J}/2 + \#\mathcal{K}/\alpha}\|A\|_{\mathcal{J}|\mathcal{K}},
\end{equation}
where for $\mathcal{J} = \{J_1,\ldots,J_r\}\in P_I$ and $\mathcal{K} = \{K_1,\ldots,K_k\} \in P_{[d]\setminus I}$,
\begin{align*}
\|A\|_{\mathcal{J}|\mathcal{K}}
=\sum_{s_1\in K_1,\ldots,s_k \in K_k}\sup\Big\{&\sum_{\ii\in[n]^d} a_{\ii}\prod_{l=1}^r x\ub{l}_{\bfi_{J_l}}\prod_{l=1}^k y\ub{l}_{\bfi_{J_l}}\colon
\|(x\ub{l}_{\bfi_{J_l}})\|_2\leq 1, \,\textrm{for $1\le l \le r$},  \\
&\sum_{i_{s_l} \le n}\|(y\ub{l}_{\bfi_{K_l}})_{\ii_{K_l \setminus \{s_l\}}}\|_2^\alpha \leq 1,\,\textrm{for $1\le l \le k$}\Big\}.
\end{align*}
If $d \le 3$, then for any $p \ge 2$,
\begin{displaymath}
C_d^{-1}m_d(p, A)\le \|\langle A,Z_1\otimes\cdots\otimes Z_d\rangle\|_p \le C_d m_d(p, A).
\end{displaymath}
Moreover, if $\alpha = 1$, then the above inequality holds for all $d \ge 1$.
\end{theorem}

Before we proceed, let us provide a few specific examples of the norms $\|A\|_{\mathcal{J}|\mathcal{K}}$, which for $\alpha < 2$ are more complicated than in the Gaussian case. In what follows, $\beta = \frac{\alpha}{\alpha -1}$ (with $\beta = \infty$ for $\alpha = 1$). For $d=1$,
\begin{align*}
  \|(a_i)\|_{\{1\}| \emptyset} &= \sup\big\{ \sum a_i x_i \colon \sum x_i^2 \le 1 \big\} = |(a_i)|_2, \\
  \|(a_i)\|_{\emptyset| \{1\}} &= \sup\big\{ \sum a_i y_i \colon \sum |y_i|^\alpha \le 1 \big\} = |(a_i)|_\beta.
\end{align*}
For $d=2$, $\|(a_{ij})\|_{\{1,2\}| \emptyset} = \|(a_{ij})\|_{\textup{HS}}$, $\|(a_{ij})\|_{\{1\}\{2\}| \emptyset} = \|(a_{ij})\|_{\ell_2 \to \ell_2}$,
\begin{align*}
  \|(a_{ij})\|_{\{1\}|\{2\}} &= \sup\big\{ \sum a_{ij} x_i y_j \colon \sum x_i^2 \le 1, \sum |y_j|^\alpha \le 1\big\} = \|(a_{ij})\|_{\ell_\alpha \to \ell_2}, \\
  \|(a_{ij})\|_{\{2\}|\{1\}} &= \sup\big\{ \sum a_{ij} y_i x_j \colon \sum x_j^2 \le 1, \sum |y_i|^\alpha \le 1\big\} = \|(a_{ij})\|_{\ell_2 \to \ell_\beta}, \\
  \|(a_{ij})\|_{\emptyset| \{1\}\{2\}} &= \sup\big\{\sum a_{ij} y_i z_j \colon \sum |y_i|^\alpha \le 1, \sum |z_j|^\alpha \le 1\big\} = \|(a_{ij})\|_{\ell_\alpha \to \ell_\beta},
\end{align*}
and
\begin{align*}
  \|(a_{ij})\|_{\emptyset| \{1,2\}} &= \sup\big\{\sum a_{ij} y_{ij} \colon \sum_i \big(\sum_j y_{ij}^2\big)^{\frac{\alpha}{2}} \le 1\big\} + \sup\big\{\sum a_{ij} y_{ij} \colon \sum_j \big(\sum_i y_{ij}^2\big)^{\frac{\alpha}{2}} \le 1\big\} \\
   &= \Big(\sum_i \big( \sum_j a_{ij}^2 \big)^{\beta/2} \Big)^{1/\beta} +
      \Big(\sum_j \big( \sum_i a_{ij}^2 \big)^{\beta/2} \Big)^{1/\beta}.
\end{align*}
For $d=3$, we have, for example,
\begin{align*}
  \|(a_{ijk})\|_{\{2\}|\{1\}\{3\}} &= \sup\big\{ \sum a_{ijk} y_i x_j z_k \colon \sum |x_j|^2 \le 1, \sum |y_i|^\alpha \le 1, \sum |z_k|^\alpha \le 1 \big\}, \\
  \|(a_{ijk})\|_{\{2\}|\{1,3\}} &= \sup\big\{ \sum a_{ijk} x_j y_{ik} \colon \sum x_j^2 \le 1, \sum_i \big( \sum_k y_{ik}^2 \big)^{\frac{\alpha}{2}} \le 1 \big\} \\
                                &+ \sup\big\{ \sum a_{ijk} x_j y_{ik} \colon \sum x_j^2 \le 1, \sum_k \big( \sum_i y_{ik}^2 \big)^{\frac{\alpha}{2}} \le 1 \big\}, \\
  \|(a_{ijk})\|_{\emptyset|\{1\} \{2,3\}} &= \sup\big\{ \sum a_{ijk} y_i z_{jk} \colon \sum |y_i|^\alpha \le 1, \sum_j \big( \sum_k z_{jk}^2 \big)^{\frac{\alpha}{2}} \le 1 \big\} \\
                                          &+ \sup\big\{ \sum a_{ijk} y_i z_{jk} \colon \sum |y_i|^\alpha \le 1, \sum_k \big( \sum_j z_{jk}^2 \big)^{\frac{\alpha}{2}} \le 1 \big\}.
\end{align*}

In particular, from Theorem \ref{thm:Weibull_chaos}  it follows  that for $\alpha \in [1,2]$, if $Y = (Y_1, \ldots, Y_n)$ is as in Theorem~\ref{thm:Weibull_chaos} then for every $x \in \R^n$,
\begin{displaymath}
\frac{1}{C}(\sqrt{p}|x|_2 + p^{1/\alpha}|x|_\beta) \le \|\langle x,Y\rangle\|_p \le C(\sqrt{p}|x|_2 + p^{1/\alpha}|x|_\beta),
\end{displaymath}
where $|\cdot|_r$ stands for $\ell_r^n$ norm (see also~\cite{GK}).
Thus, for $\beta \in (2, \infty)$, the inequality of Theorem \ref{thm:AidaStroock}, for $m=n$, $k = 1$ and a $\mathcal{C}^1$ function $f \colon \R^n \to \R$, can be written in the form
\begin{equation}\label{ineq:sobolev-beta}
\|f(X) - \E f(X)\|_p \le C_\beta \|\langle \nabla f(X),Y\rangle\|_p.
\end{equation}
This allows for induction, just as in the proof of Proposition \ref{prop:moment_Poincare}, except that instead of Gaussian vectors we will have independent copies of $Y$. We can thus repeat the proof of Theorem \ref{thm:main}, using the above observation and Theorem \ref{thm:Weibull_chaos} instead of Theorem \ref{thm_Latala_dec}. This argument will then yield the following proposition, which is a counterpart of Theorem \ref{thm:main}.
At the moment we can prove it only for $D \le 3$, clearly generalizing Theorem \ref{thm:Weibull_chaos} to chaoses of arbitrary degree would immediately imply it for general $D$.

\begin{prop}\label{prop:Weibull} Let $X = (X_1,\ldots,X_n)$ be a random vector in $\R^n$, with independent components. Let $\beta \in (2,\infty)$ and assume that for all $i \le n$, $X_i$ satisfies the $\beta$-modified logarithmic Sobolev inequality with constant $D_{LS_\beta}$. Let $f \colon \R^n \to \R$ be a $\mathcal{C}^D$ function. Define
\[
m(p, f) = \big\| m_D(p, \D^D f(X)) \big\|_p + \sum_{1 \le d \le D-1} m_d(p, \E \D^d f(X)),
\]
where $m_d(p, A)$ is defined by~\eqref{eq:mdpA} with $\alpha = \frac{\beta}{\beta-1}$.
If $D\le 3$ then for $p\ge 2$,
\begin{displaymath}
\|f(X)-\E f(X)\|_p \le C_{\beta,D_{LS_\beta}} m(p, f).
\end{displaymath}
As a consequence, for all $p \ge 2$,
\begin{displaymath}
\p\big(|f(X) - \E f(X)| \ge C_{\beta,D_{LS_\beta}} m(p, f)\big) \le e^{-p}.
\end{displaymath}
\end{prop}
\paragraph{Remarks}
\paragraph{1.} For $\beta = 2$, the estimates of the above proposition agree with those of Theorem \ref{thm:main_intro}. For $\beta > 2$ it improves on what can be obtained from Theorem \ref{thm:main} in two aspects (of course just for $D \le 3$). First, the exponent of $p$ is smaller as $(\gamma-1/2)d +\#(\mathcal{J}\cup\mathcal{K})/2 = (1/\alpha -1/2)d+\#\mathcal{J}/2 + \#\mathcal{K}/2 \ge \#\mathcal{J}/2 + \#\mathcal{K}/\alpha$. Second $\|A\|_{\mathcal{J}\cup \mathcal{K}} \ge \|A\|_{\mathcal{J}|\mathcal{K}}$ (since for $\alpha < 2$, $|x|_\alpha \ge |x|_2$, so the supremum on the left hand side is taken over a larger set).

\paragraph{2.} From results in \cite{Chaos3d} it follows that if $f$ is a tetrahedral polynomial of degree $D$ and $X_i$ are i.i.d. symmetric random variables satisfying $\p(|X_i| \ge t) = \exp(-t^\alpha)$, then the inequalities of Proposition \ref{prop:Weibull} can be reversed (up to constants), i.e.,
\begin{displaymath}
\|f(X) - \E f(X)\|_p \ge \frac{1}{C_D}m_f(p).
\end{displaymath}
This is true for any positive integer $D$.

\paragraph{3.} One can also consider another functional inequality, which may be regarded a counterpart of \eqref{eq:modifiedLS} for $\beta = \infty$. We say that a random vector $X$ in $\R^n$ satisfies the Bobkov-Ledoux inequality if for all locally Lipschitz positive functions such that $|\nabla f(x)|_\infty := \max_{1\le i \le n} |\frac{\partial }{\partial x_i} f(x)| \le d_{BL}f(x)$ for all $x$,
 \begin{align}\label{eq:BL}
 \Ent f^2(X) \le D_{BL} \E |\nabla f(X)|^2.
 \end{align}
 This inequality has been introduced in \cite{BobLed_exp} to provide a simple proof of Talagrand's two-level concentration for the symmetric exponential measure in $\R^n$. Here $|\frac{\partial }{\partial x_i} f(x)|$ is defined as ``partial length of gradient'' (see \eqref{eq:length_of_gradient}). Thus in the case of differentiable functions $|\nabla f|_\infty$ coincides with the $\ell_\infty^n$ norm of the ``true'' gradient.

In view of Theorem \ref{thm:AidaStroock} it is natural to conjecture that the Bobkov-Ledoux inequality implies
\begin{align}\label{eq:Sobolev_exp}
\|f(X) - \E f(X)\|_p \le C\Big(\sqrt{p}\big\||\nabla f(X)|\big\|_p + p\big\| |\nabla f(X)|_\infty \big\|_p\Big),
\end{align}
which in turn implies~\eqref{ineq:sobolev-beta} with $Y = (Y_1, \ldots, Y_n)$ being a vector of independent symmetric exponential variables and some $C_\infty < \infty$. This would yield an analogue of Proposition~\ref{prop:Weibull} for $\beta = \infty$, this time with no restriction on $D$.

Unfortunately at present we do not know whether the implication \eqref{eq:BL} $\implies$ \eqref{eq:Sobolev_exp} holds true or even if  \eqref{eq:Sobolev_exp} holds for the symmetric exponential measure in $\R^n$.
We only are able to prove the following weaker inequality, which is however not sufficient to obtain a counterpart of Proposition \ref{prop:Weibull} for $\beta = \infty$.

\begin{prop}\label{prop:exp}
If $X$ is a random vector in $\R^n$, which satisfies \eqref{eq:BL}, then for any locally Lipschitz function $f \colon \R^n \to \R$, and any $p \ge 2$,
\begin{displaymath}
\|f(X) - \E f(X)\|_p \le 3\Big(D_{BL}^{1/2} \sqrt{p}\big\||\nabla f(X)|\big\|_p + d_{BL}^{-1}p \big\||\nabla f(X)|_\infty\big\|_\infty\Big).
\end{displaymath}
\end{prop}
\begin{proof}
To simplify the notation we suppress the argument $X$. In what follows $\|\cdot\|_p$ denotes the $L_p$ norm with respect to the distribution of $X$.

Let us fix $p \ge 2$ and consider $f_1 = \max(f, \|f\|_p /2)$. We have
\begin{align}\label{eq:properties_f_1}
\|f_1\|_p &\ge \|f\|_p,\\
\|f_1\|_2 &\le \frac12 \|f\|_p + \|f\|_2,\nonumber\\
\|f_1\|_p &\le \frac{3}{2}\|f\|_p \le 3 \min f_1.\nonumber
\end{align}

Moreover, $f_1$ is locally Lipschitz and we have pointwise estimates $|\nabla f_1|\le |\nabla f|$, $|\nabla f_1|_\infty \le |\nabla f|_\infty$.
Assume now that we have proved that
\begin{align}\label{eq:BL_auxiliary}
\|f_1\|_p \le \|f_1\|_2 + \sqrt{\frac{D_{BL}}{2}}\sqrt{p}\big\||\nabla f_1| \big\|_p + \frac{3p}{2d_{BL}}\big\||\nabla f_1|_\infty\big\|_\infty.
\end{align}
Then, together with the two first inequalities of~\eqref{eq:properties_f_1}, it yields
\begin{align*}
\|f\|_p &\le \|f_1\|_p \le \|f_1\|_2 + \sqrt{\frac{D_{BL}}{2}}\sqrt{p}\big\||\nabla f_1| \big\|_p + \frac{3p}{2d_{BL}}\big\||\nabla f_1|_\infty\big\|_\infty\\
& \le \frac{1}{2}\|f\|_p + \|f\|_2 + \sqrt{\frac{D_{BL}}{2}}\sqrt{p}\big\||\nabla f| \big\|_p + \frac{3p}{2d_{BL}}\big\||\nabla f|_\infty\big\|_\infty,
\end{align*}
which gives
\begin{equation}\label{eq:BL_fp}
\|f\|_p \le 2\Big(\|f\|_2 + \sqrt{\frac{D_{BL}}{2}}\sqrt{p}\big\||\nabla f| \big\|_p + \frac{3p}{2d_{BL}}\big\||\nabla f|_\infty\big\|_\infty\Big).
\end{equation}
Since \eqref{eq:BL} implies the Poincar\'{e} inequality with constant $D_{BL}/2$ (see e.g. Proposition 2.3 in \cite{Gentil-Guillin-Miclo-2005}), we can conclude the proof applying~\eqref{eq:BL_fp} to $|f - \E f|$ (similarly as in the proof of Theorem \ref{thm:AidaStroock}). Thus it is enough to prove \eqref{eq:BL_auxiliary}.

From now on we are going to work with the function $f_1$ only, so for brevity we will drop the subscript and write $f$ instead of $f_1$. Assume $\|f\|_p \ge \frac{3p}{2d_{BL}} \||\nabla f|_\infty \|_\infty$ (otherwise~\eqref{eq:BL_auxiliary} is trivially satisfied). Then, using the third inequality of \eqref{eq:properties_f_1}, for $2\le t \le p$ and all $x \in \R^n$,
\begin{displaymath}
|\nabla f^{t/2}(x)|_\infty \le \frac{t}{2} f^{t/2-1}(x) |\nabla f(x)|_\infty \le \frac{3}{2} f^{t/2}(x) \frac{p|\nabla f(x)|_\infty}{\|f\|_p} \le d_{BL} f^{t/2}(x).
\end{displaymath}
We can thus apply~\eqref{eq:BL} with $f^{t/2}$, which together with H\"older's inequality gives
\begin{displaymath}
\Ent f^t \le D_{BL} \E |\nabla f^{t/2}|^2 \le D_{BL}\frac{t^2}{4}\E \big( f^{t-2} |\nabla f|^2 \big) \le D_{BL}\frac{t^2}{4} \big\||\nabla f|\big\|_t^2 (\E f^t)^{1 - \frac{2}{t}}.
\end{displaymath}
Now, as in the proof of Theorem \ref{thm:AidaStroock}, we have
\begin{displaymath}
\frac{d}{dt} (\E f^t)^{2/t} = \frac{2}{t^2} (\E f^t)^{\frac{2}{t} - 1} \Ent f^t \le \frac{D_{BL}}{2}\big\||\nabla f|\big\|_p^2,
\end{displaymath}
which upon integrating gives
\begin{displaymath}
\|f\|_p^2 \le \|f\|_2^2 + \frac{D_{BL}}{2} p \big\||\nabla f|\big\|_p^2,
\end{displaymath}
which clearly implies \eqref{eq:BL_auxiliary}.
\end{proof}

\section{Appendix}

\subsection{Decoupling inequalities}
Let us here state the main decoupling result for $U$-statistics (Theorem 1 in \cite{dlPMS1}).

\begin{theorem}\label{thm:decoupling} For natural numbers $n \ge d$ let $(X_i)_{i=1}^n$ be a sequence of independent random variables with values in a measurable space $(S,\mathcal{S})$ and let $(X\ub{j}_i)_{i=1}^n$ $j= 1,\ldots,d$ be $d$ independent copies of this sequence. Let $B$ be a separable Banach space and for each $\ii \in [n]\uu{d}$ let $h_\ii\colon S^d \to B$ be a measurable function. Then for all $t > 0$,
\begin{displaymath}
\p\Big(\Big\|\sum_{\ii \in [n]\uu{d}} h_\ii (X_{i_1},\ldots,X_{i_d})\Big\|> t\Big) \le C_d \p\Big(\Big\|\sum_{\ii \in [n]\uu{d}} h_\ii (X\ub{1}_{i_1},\ldots,X\ub{d}_{i_d})\Big\|> t/C_d\Big).
\end{displaymath}
In consequence for all $p \ge 1$,
\begin{displaymath}
\Big\|\sum_{\ii \in [n]\uu{d}} h_\ii (X_{i_1},\ldots,X_{i_d})\Big\|_p \le C_d \Big\|\sum_{\ii \in [n]\uu{d}} h_\ii (X\ub{1}_{i_1},\ldots,X\ub{d}_{i_d})\Big\|_p.
\end{displaymath}

If moreover the functions $h_\ii$ are symmetric in the sense that, for all $x_1,\ldots,x_d \in S$ and all permutations $\pi\colon [d]\to [d]$,
$h_{i_1,\ldots,i_d}(x_1,\ldots,x_d) = h_{i_{\pi_1},\ldots,i_{\pi_d}}(x_{\pi_1},\ldots,x_{\pi_d})$,
then for all $t > 0$,
\begin{displaymath}
\p\Big(\Big\|\sum_{\ii \in [n]\uu{d}} h_\ii (X\ub{1}_{i_1},\ldots,X\ub{d}_{i_d})\Big\|> t\Big) \le C_d \p\Big(\Big\|\sum_{\ii \in [n]\uu{d}} h_\ii (X_{i_1},\ldots,X_{i_d})\Big\|> t/C_d\Big)
\end{displaymath}
and in  consequence for all $p \ge 1$,
\begin{displaymath}
\Big\|\sum_{\ii \in [n]\uu{d}} h_\ii (X\ub{1}_{i_1},\ldots,X\ub{d}_{i_d})\Big\|_p \le C_d\Big\|\sum_{\ii \in [n]\uu{d}} h_\ii (X_{i_1},\ldots,X_{i_d})\Big\|_p.
\end{displaymath}

\end{theorem}

\subsection{Proof of Lemma \ref{lemma_comp_Gauss}}

Without loss of generality we can assume that $M = 1$. It is easy to see that for some constant $C_k$ and $t > 1$, $\p(C_k|g_{i1}\cdots g_{ik}| > t) \ge 2\exp(-t^{2/k})$. Since $ \p(|Y_i|\ge t) \le 2\exp(-t^{2/k})$, we get
\begin{displaymath}
\p(|Y_i\Ind{|Y_i|\ge 1}| \ge t)  \le \p(C_k |g_{i1}\cdots g_{ik}| > t).
\end{displaymath}

Therefore, using the inverse of the distribution function, we can define i.i.d. copies $\tilde{Y}_i$ of $|Y_i\Ind{|Y_i|\ge 1}|$ and i.i.d copies $Z_i$ of $|g_{i1}\cdots g_{ik}|$, such that $\tilde{Y}_i \le C_k Z_i$ pointwise. We may assume that these copies are defined on a common probability space with $Y_i$ and $g_{ij}$. We can now write for a sequence $\varepsilon_i$ of i.i.d. Rademacher variables independent of all he variables introduced so far.

\begin{align*}
\|\sum_{i=1}^n a_i Y_i\|_p & = \|\sum_{i=1}^n a_i \varepsilon_i|Y_i|\|_p\\
 &\le \|\sum_{i=1}^n a_i \varepsilon_i|Y_i\ind{|Y_i|< 1}|\|_p + \|\sum_{i=1}^n a_i \varepsilon_i|Y_i\ind{|Y_i|\ge 1}|\|_p\\
&=  \|\sum_{i=1}^n a_i \varepsilon_i|Y_i\ind{|Y_i|< 1}|\|_p + \|\sum_{i=1}^n a_i \varepsilon_i\tilde{Y}_i\|_p\\
&\le \|\sum_{i=1}^n a_i \varepsilon_i\|_p + C_k\|\sum_{i=1}^n a_i \varepsilon_i Z_i\|_p\\
&\le C_k \|\sum_{i=1}^n a_i \varepsilon_i \E_Z Z_i\|_p + C_k\|\sum_{i=1}^n a_i \varepsilon_iZ_i\|_p\\
&\le C_k \|\sum_{i=1}^n a_i \varepsilon_i Z_i\|_p = C_k\|\sum_{i=1}^n a_i g_{i_1}\cdots g_{ik}\|_p,
\end{align*}
where in the second inequality we used the contraction principle (once conditionally on $\tilde{Y}_i$'s and $Z_i$'s) and in the third one Jensen's inequality.
\bibliographystyle{abbrv}

\bibliography{citations}

\bigskip

\noindent E-mails: \texttt{R.Adamczak@mimuw.edu.pl}, \texttt{P.Wolff@mimuw.edu.pl}

\end{document}